\documentclass{amsart}
\usepackage[utf8]{inputenc}

\usepackage{mathpazo}
\usepackage{hyperref}
\hypersetup{%
  colorlinks,               
  linkcolor={red!50!black}, 
  citecolor={blue!50!black},
  urlcolor={blue!80!black}  
}

\usepackage{amssymb}
\usepackage{amsmath}
\usepackage{amsfonts}
\usepackage{amsthm}

\usepackage{relsize}

\usepackage[final]{graphicx}
\usepackage{epstopdf}

\usepackage{indentfirst}
\usepackage{lipsum}

\usepackage{marvosym}
\usepackage{booktabs}
\usepackage{lscape}

\usepackage{mathtools}

\usepackage{array}

\usepackage{multirow}
\usepackage{multicol}

\usepackage[backgroundcolor=white,linecolor=lime,bordercolor=white,figcolor=white]{todonotes}
\makeatletter
\providecommand\@dotsep{5}\renewcommand{\listoftodos}[1][\@todonotes@todolistname]{%
  \@starttoc{tdo}{#1}}
\makeatother
\usepackage{enumerate}

\usepackage{xcolor}

\usepackage{iftex}

\usepackage{cleveref}
\usepackage{cancel}
\usepackage{diagbox}

\newcommand{\Ext}{\operatorname{Ext}}
\newcommand{\supp}{\operatorname{supp}}

\newcommand{\Hom}{\operatorname{Hom}}
\newcommand{\End}{\operatorname{End}}
\newcommand{\im}{\operatorname{im}}
\newcommand{\id}{{\operatorname{id}}}

\newcommand{\rk} {\operatorname{rk}\,}

\newcommand{\cat}[1]{\underline{#1}}
\newcommand{\cmod}[1]{\cat{\text{mod}\,#1}}
\newcommand{\Z}{\mathbb{Z}}
\newcommand{\N}{\mathbb{N}}
\newcommand{\C}{\mathbb{C}}

\newcommand{\Q}{\mathbb{Q}}

\newcommand{\mathscr}{\mathcal}

\definecolor{rdylgn60}{HTML}{D73027}
\definecolor{rdylgn61}{HTML}{FC8D59}
\definecolor{rdylgn62}{HTML}{FEE08B}
\definecolor{rdylgn63}{HTML}{D9EF8B}
\definecolor{rdylgn64}{HTML}{91CF60}
\definecolor{rdylgn65}{HTML}{1A9850}
\definecolor{rdylgn66}{HTML}{006837}
\definecolor{ca1}{HTML}{72497F}
\definecolor{ca2}{HTML}{5C7B78}
\definecolor{cb1}{HTML}{0862A0}
\definecolor{cb2}{HTML}{CD9B3E}
\definecolor{cc1}{HTML}{C55824}
\definecolor{cc2}{HTML}{288338}
\definecolor{cd}{HTML}{CF838A}
\definecolor{ce}{HTML}{FFAF3A}
\definecolor{cf}{HTML}{AFBF5A}
\definecolor{cg}{HTML}{7ACF99}
\definecolor{ch}{HTML}{8AB1CF}
\definecolor{ci}{HTML}{3B58CF}
\definecolor{cj}{HTML}{736163}

\definecolor{lightgray}{HTML}{DDDDDD}

\makeatletter
\@ifclassloaded{beamer}{%
}{%
\usepackage[margin=1.5in]{geometry}
\makeatother
\theoremstyle{plain}
\newtheorem{theorem}{Theorem}[section]
\newtheorem{proposition}[theorem]{Proposition}
\newtheorem{definition}[theorem]{Definition}
\newtheorem{example}[theorem]{Example}
\newtheorem{lemma}[theorem]{Lemma}

\newtheorem{corollary}[theorem]{Corollary}
}
\newtheorem{conjecture}{Conjecture}
\theoremstyle{remark}
\newtheorem{remark}{Remark}[section]

\usepackage{genyoungtabtikz}

\numberwithin{equation}{section}

\usepackage{tikz}
\usetikzlibrary{shapes,arrows,calc}

\usetikzlibrary{decorations.pathreplacing,decorations.pathmorphing}

\usepackage{tikz-cd}
\ifLuaTeX%
\usetikzlibrary{graphdrawing}
\usetikzlibrary{graphdrawing.layered}
\fi

\tikzset{%
  spot/.style={color=black, thin, dashed},
  rline/.style={color=green, line width=2pt},
  sline/.style={color=blue, line width=2pt},
  tline/.style={color=red, line width=2pt},
  uline/.style={color=green, line width=2pt},
  line/.style={color=#1, line width=2pt},
  line/.default=blue,
  rdot/.style={color=green, thin, fill},
  sdot/.style={color=blue, thin, fill},
  tdot/.style={color=red, thin, fill},
  udot/.style={color=green, thin, fill},
  dot/.style={color=#1, thin, fill},
  dot/.default=blue
}

\newcommand{\namedses}[5]{{#3}\stackrel{#1}{\hookrightarrow}{#4}\stackrel{#2}{\twoheadrightarrow}{#5}}

\newcommand{\rt}{{\color{red} t}}
\newcommand{\rs}{{\color{blue} s}}

\newcommand{\TL}{{\rm TL}}
\newcommand{\JW}{{\rm JW}}

\newcommand{\TLcat}{{\mathscr{TL}}}

\newcommand{\gaussianquant}{\genfrac{[}{]}{0pt}{}}

\newcommand{\restr}{\mathord\downarrow}

\let\emph\relax 
\DeclareTextFontCommand{\emph}{\bfseries\em}

\author[R. A. Spencer]{Robert A. Spencer}
\address{Department of Pure Mathematics and Mathematical Statistics\\ University of Cambridge \\Cambridge\\ CB3 0WA \\ United Kingdom}
\email{ras228@cam.ac.uk}
\date{}

\graphicspath{{diagrams/}}
\title{The Modular Temperley-Lieb Algebra}
\begin{document}

\begin{abstract}
  We investigate the representation theory of the Temperley-Lieb algebra, $\TL_n(\delta)$, defined over a field of positive characteristic.
  The principle question we seek to answer is the multiplicity of simple modules in cell modules for $\TL_n$ over arbitrary fields.
  This provides us with the decomposition numbers for this algebra, as well as the dimensions of all simple modules.
  We obtain these results from purely diagrammatic principles, without appealing to realisations of $\TL_n$ as endomorphism algebras of $U_q(\mathfrak{sl}_2)$ modules.
  Our results strictly generalise the known characteristic zero theory of the Temperley-Lieb algebras.
\end{abstract}

\maketitle

\section*{Introduction}\label{sec:introduction}
The defining relations for the Temperley--Lieb algebras were first introduced by Temperley and Lieb in 1971, in order to study linear statistical mechanics problems of the ``Potts'' or ``ice'' type\cite{temperley_lieb_1971}.
The formulation given is in terms of transfer matrices which act on electron-spin state space and the operators are defined in terms of their action on the state space.
The defining relations of
\begin{eqnarray}
  u_i^2 &=& \left(\frac{r}{s} + \frac{s}{r}\right) u_i\\
  u_i u_{i \pm 1} u_i &=& u_i \\
  u_i u_j &=& u_j u_i \quad\quad\quad\quad\quad\quad |i-j|>1
\end{eqnarray}
for complex scalars $r$ and $s$, are given in terms of this action on what are now called ``cell modules''.

The Temperley--Lieb algebras not only admit a simple generators and relations description, but are a quotient of the Hecke algebra of type $A_n$, are considered a canonical example of a cellular algebra, and more recently can be phrased as a simple example of a $2$-category.
They have thus been extensively studied in the literature.
The representation theory is well understood in characteristic zero~\cite{goodman_wenzl_1993,ridout_saint_aubin_2014,westbury_1995},
and the algebras are known to be semi-simple for generic parameter (over any ring), but to have more intricate behaviour when specialised at a root of unity.
Their representation theory can be described by recursively defined ``path idempotents''~\cite{goodman_wenzl_1993} and the first critical Jones-Wenzl idempotents are known in a closed form~\cite{graham_lehrer_1998,morrison_2014}.
The interested reader is directed to the paper by Ridout and Saint-Aubin~\cite{ridout_saint_aubin_2014} for an excellent and comprehensive treatment of the representation theory in characteristic zero.
A recursive formula for the dimensions of the simple modules and Alperin diagrams for the projective indecomposable modules are provided.

More recently, attention has been given to the problem of determining the corresponding results over positive characteristic~\cite{andersen_2019,burrull_libedinsky_sentinelli_2019,cox_graham_martin_2003}.
Historically, the principal tool in this effort has been the Schur-Weyl duality that exists between $\TL_n(q+q^{-1})$ and $U_q(\mathfrak{sl}_2)$, as the representation theory of the latter is well understood~\cite{andersen_catharina_tubbenhauer_2018,cox_erdmann_2000,donkin_1998}.
This realises $\TL_n$ as the endomorphism ring of the $n$-fold tensor iterate of the standard representation $\End_{U_q(\mathfrak{sl}_2)}(V^{\otimes n})$.
Here the study of the Weyl and tilting modules of $U_q(\mathfrak{sl}_2)$ can be linked to the study of cell modules of $\TL_n$.

Recently, characteristic zero Temperley--Lieb algebras have found application to Soergel bimodule theory, where close cousins are to be found as degree zero morphisms between Bott-Samelson modules associated to the dihedral group~\cite{elias_2016}.
Here the degree $m$ of the dihedral group $D_{2m}$ and assumptions about the realisation place restrictions on the parameter of $\TL_n$.
The Jones-Wenzl idempotents describe the indecomposable Soergel bimodules.

If taken over positive characteristic, the theory describes a novel basis of the Hecke algebra associated to the dihedral group.
This ``$p$-canonical basis'' can be computed using local intersection forms.
The theory is explored by Jensen and Williamson for crystallographic Coxeter systems in~\cite{jensen_williamson_2015}.
The simplest non-crystallographic system to explore is the dihedral group and so exploration of the intersection forms (or Gram matrices of cell modules for $\TL_n$) provides a generalisation of the theory to new Coxeter systems.
In this case the desired result is the dimension of the simple $\TL_n$ modules over positive characteristic.

However, throughout, the Temperley--Lieb algebras persist in admitting a simple, generators and relations, inherently diagrammatic, definition.
In this way, they can be studied ``on their own'' without remit to endomorphism algebras of tensor iterates of $U_q(\mathfrak{sl}_2)$ modules or as part of a larger category of bimodules.
Further, this formalism is slightly more general as it makes no constraint on the parameter being of the form $q + q^{-1}$ (in particular, the parameter $\delta$ may not admit square roots of $\delta^2-4$ in the ring) and so the algebras can be defined over any pointed ring.
Further, considering the algebras diagrammatically highlights their cellular structure and opens the door to studying related cellular algebras (such as their Hecke algebras) from first principles.

In this paper, we answer the foremost question of modular representation theory, the decomposition numbers, of the Temperley--Lieb algebra over an arbitrary field, with arbitrary parameter.
This gives rise to a closed form for the dimensions of the simple modules for all $\TL_n$ modules and answers a question of when Jones-Wenzl idempotents descend to characteristic $p$ exactly.

\vspace{1em}

\noindent
The remainder of the paper is arranged as follows.
In \cref{sec:definitions_and_preliminary_results} we define the diagrammatic Temperley-Lieb category and its basic properties.
In \cref{sec:quantum_numbers} the ring over which the category is defined is considered, and quantum numbers defined.  These numbers will appear throughout the theory of the Temperley-Lieb algebra.

In \cref{sec:cellular,sec:cell-modules} we review the basic cellular theory and how it applies to the Temperley-Lieb case.  The known form of the Gram determinant is presented and interrogated in cases that will be useful in later proofs.
A truncation functor is introduced in \cref{sec:truncation} which reduces the problem to that of finding the multiplicity of the trivial module.
We can further restrict our attention by the character of $F_n$, a central element introduced by~\cite{ridout_saint_aubin_2014} that gives a partitioning of the modules slightly coarser than blocks.  We examine this in the modular case in \cref{sec:linkage}.

In \cref{sec:morphisms} we show that morphisms between cell modules are unique,up to scalar, and use this to show the main result (the decomposition numbers) in \cref{sec:decomposition}.
Inverting the decomposition matrix to obtain a closed form for the simple dimensions is performed in \cref{sec:dimensions}.
Finally, we examine some applications and prospects for the results obtained in \cref{sec:applications}.

\section{Definitions and Preliminary Results}\label{sec:definitions_and_preliminary_results}
We will define the Temperley--Lieb category diagrammatically.

\begin{definition}
  For natural numbers $n$ and $m$ of the same parity, an $(n,m)$-diagram is formed of $n+m$ points arranged in two vertical columns of $n$ and $m$ points each, paired by $(n+m)/2$ uncrossing lines in the strip between the columns.
  Two diagrams are equivalent if one may be continuously deformed into the other without moving endpoints.
\end{definition}
Examples of diagrams are:

\vspace{1em}
  \begin{tabular}{ccccc}
    \raisebox{6pt}{
      \begin{tikzpicture}[scale=0.5]
  \draw (0.0, 0.5) -- (0.0,3.5);
  \draw (2.0, -.5) -- (2.0,4.5);
  \foreach \i in {0,...,2} {
    \draw[fill] (0, \i+1) circle (0.1);
  }
  \foreach \i in {0,...,4} {
    \draw[fill] (2, \i) circle (0.1);
  }
  \draw[very thick] (0, 1) to[out=0, in=180] (2,0);
  \draw[very thick] (0, 2) to[out=0, in=180] (2,3);
  \draw[very thick] (0, 3) to[out=0, in=180] (2,4);

  \draw[very thick] (2, 1) to[out=180, in=180] (2,2);
\end{tikzpicture}
    }
    &&
    \raisebox{12pt}{
      \begin{tikzpicture}[scale=0.5]
  \draw (0.0, -.5) -- (0.0,3.5);
  \draw (2.0, -.5) -- (2.0,3.5);
  \foreach \i in {0,...,3} {
    \draw[fill] (0, \i) circle (0.1);
    \draw[fill] (2, \i) circle (0.1);
  }
  \draw[very thick] (0, 0) to[out=0, in=180] (2,2);
  \draw[very thick] (0, 1) to[out=0, in=0] (0,2);
  \draw[very thick] (0, 3) to[out=0, in=180] (2,3);

  \draw[very thick] (2, 0) to[out=180, in=180] (2,1);
\end{tikzpicture}
    }
    &&
    \raisebox{0pt}{
      \begin{tikzpicture}[scale=0.5]
  \draw (0.0, -.5) -- (0.0,5.5);
  \draw (2.0, -.5) -- (2.0,5.5);
  \foreach \i in {0,...,5} {
    \draw[fill] (0, \i) circle (0.1);
    \draw[fill] (2, \i) circle (0.1);
  }
  \draw[very thick] (0, 0) to[out=0, in=0] (0,1);
  \draw[very thick] (0, 2) to[out=0, in=0] (0,3);
  \draw[very thick] (0, 4) to[out=0, in=0] (0,5);

  \draw[very thick] (2, 0) to[out=180, in=180] (2,3);
  \draw[very thick] (2, 2) to[out=180, in=180] (2,1);
  \draw[very thick] (2, 4) to[out=180, in=180] (2,5);
\end{tikzpicture}
    }
    \\
    An epic $(3,5)$-diagram&
    \hspace{6em}&
    A $(4,4)$-diagram&
    \hspace{6em}&
    A $(6,6)$-diagram\\
  \end{tabular}
\vspace{1em}

For an $(n,m)$-diagram, the $n$ points on the left are known as \emph{source sites} and those on the right as \emph{target sites}.
When counting sites, we will enumerate them from top to bottom.
The number of source sites a diagram connects to target sites is the propagation number of the diagram.
Diagrams for which the propagation number is maximal (i.e.\ the larger of $n$ and $m$) are either \emph{monic} if $m \leq n$ or \emph{epic} if $m \ge n$.

The monic $(n,n-2)$-diagram for which source site $i$ is connected to source site $i+1$ is known as the \emph{simple cap} at $i$.  The corresponding $(n-2, n)$-diagram is known as the \emph{simple cup}.  We will denote the simple $(2,0)$-cap by $\cap$ and the $(0,2)$-cup by $\cup$.

We now form a category from these diagrams.
Let $R$ be a commutative ring and $\delta$ a distinguished element.
We will always be assuming that $R$ is commutative.
We may refer to the ``pointed ring'' $(R,\delta)$ or simply $R$ if the value of $\delta$ is clear.
\begin{definition}
  The Temperley--Lieb category, $\TLcat^R(\delta)$ is an abelian category over $R$ with object set $\{\underline{n} : n\in \N\}$.
  The space of morphisms $\underline{n} \to \underline{m}$ has basis given by $(n,m)$-diagrams.
  Composition is defined on this basis by identification of the appropriate source and target sites with each closed loop resolving to a linear factor of $\delta$.
\end{definition}
Throughout, where $R$ or $\delta$ are unambiguous or arbitrary, we will omit them from the notation.

To illustrate morphism composition, we present a composition of a (5,7)-diagram in $\Hom_{\TLcat}(\underline{5}, \underline{7})$ with a (7,3)-diagram in $\Hom_{\TLcat}(\underline{7}, \underline{3})$.
The resultant morphism from $\underline{5} \to \underline{3}$ is computed by identifying the seven points on the right of the first morphism with the seven from the left of the second and resolving the single resulting closed loop to a factor of $\delta$.

\vspace{1em}
\begin{center}
  \begin{tikzpicture}[scale=0.5]
    \draw (0, -0.5) -- (0,4.5);
    \draw (2, -1.5) -- (2,5.5);
    \draw (4, -1.5) -- (4,5.5);
    \draw (6, .5) -- (6,3.5);
    \draw (8.75, -0.5) -- (8.75,4.5);
    \draw (10.75, .5) -- (10.75,3.5);
    \foreach \i in {0,...,4} {
      \draw[fill] (0, \i) circle (0.1);
      \draw[fill] (8.75, \i) circle (0.1);
    };
    \foreach \i in {0,...,6} {
      \draw[fill] (2, \i-1) circle (0.1);
      \draw[fill] (4, \i-1) circle (0.1);
    };
    \foreach \i in {0,...,2} {
      \draw[fill] (6, \i+1) circle (0.1);
      \draw[fill] (10.75, \i+1) circle (0.1);
    };
    \node at (3, 2) {$\circ$};
    \node at (7, 2) {$=$};
    \node at (8., 2) {$\delta$};

    \draw[very thick] (2, 0) to[out=180, in=180] (2,-1);
    \draw[very thick] (4, 0) to[out=0, in=0] (4,-1);
    \draw[very thick] (6, 2) to[out=180, in=180] (6,1);

    \draw[very thick] (0, 0) to[out=0, in=180] (2,1);
    \draw[very thick] (0, 2) to[out=0, in=0] (0,3);
    \draw[very thick] (0, 1) to[out=0, in=180] (2,4);
    \draw[very thick] (0, 4) to[out=0, in=180] (2,5);
    \draw[very thick] (2, 2) to[out=180, in=180] (2,3);

    \draw[very thick] (4, 1) to[out=0, in=0] (4,2);
    \draw[very thick] (4, 4) to[out=0, in=0] (4,5);
    \draw[very thick] (4, 3) to[out=0, in=180] (6,3);

    \draw[very thick] (8.75, 0) to[out=0, in=180] (10.75,3);
    \draw[very thick] (8.75, 1) to[out=0, in=0] (8.75,4);
    \draw[very thick] (8.75, 2) to[out=0, in=0] (8.75,3);
    \draw[very thick] (10.75, 2) to[out=180, in=180] (10.75,1);
  \end{tikzpicture}
\end{center}
\vspace{1em}

It is clear that there are no non-zero morphisms between objects $\underline{n}$ and $\underline{m}$ when $n\not\equiv_2 m$.
Further, monic diagrams are monomorphisms and epic diagrams epimorphisms in this category.

This category, $\TLcat$ is actually monoidal, where the tensor product $- \otimes -$ sends $\underline{n}\otimes \underline{m} = \underline{n+m}$ and acts on morphisms by vertical concatenation.

The \emph{support} of a morphism is the set of diagrams appearing with non-zero corresponding coefficient if the morphism is written in the diagram basis.

\begin{definition}
  The Temperley--Lieb algebra on $n$ sites, $\TL^R_n(\delta)$ is $\End_{\TLcat^R(\delta)}(\underline{n})$.
\end{definition}
As previously mentioned, if $R$ or $\delta$ are understood or arbitrary, we will omit them from the notation.  It is clear that $\TL_n$ is a unital algebra with unit the unique diagram with propagation number $n$.
It admits a description in terms of generators and relations as an algebra with generators $\{u_i : 1\le i < n\}$ subject to
\begin{align*}
  u_i^2 &= \delta u_i\\
  u_iu_j &= u_j u_i & |i - j| \ge 2\\
  u_iu_{i\pm1}u_i &= u_i
\end{align*}
Here, $u_i$ corresponds to the diagram with a simple cup and simple cap at $i$.
\vspace{1em}
\begin{center}
  \begin{tabular}{c}
    \begin{tabular}{l}
      \begin{tikzpicture}[scale=0.5]
  \draw (0.0, -.5) -- (0.0,4.5);
  \draw (2.0, -.5) -- (2.0,4.5);
  \foreach \i in {0,...,4} {
    \draw[fill] (0, \i) circle (0.1);
    \draw[fill] (2, \i) circle (0.1);
  }
  \draw[very thick] (0, 0) to[out=0, in=180] (2,0);
  \draw[very thick] (0, 1) to[out=0, in=180] (2,1);
  \draw[very thick] (0, 4) to[out=0, in=180] (2,4);

  \draw[very thick] (0, 2) to[out=0, in=0] (0,3);
  \draw[very thick] (2, 2) to[out=180, in=180] (2,3);
\end{tikzpicture}
    \end{tabular}
    \\
    The $(5,5)$-diagram, $u_2$\\
  \end{tabular}
\end{center}
\vspace{1em}

For diagram $\underline{n}\to\underline{m}$, the propagation number is also the minimum $k$ such that the corresponding morphism factors through $\underline{k}$.
The propagation number of a morphism is the maximum of the propagation numbers of the diagrams in its support.
As such composition cannot increase propagation number and so we have a strict filtration of ideals in $\TL_n$,
\begin{equation}\label{eq:tln_filt}
  \TL_n = \mathcal{F}^n(\TL_n) \supset
  \mathcal{F}^{n-2}(\TL_n) \supset
  \cdots \supset
  \mathcal{F}^{n- 2\lfloor n/2\rfloor}(\TL_n) \supset 0,
\end{equation}
where $\mathcal{F}^i(\TL_n)$ contains all morphisms of propagation number at most $i$.

\begin{definition}
  A standard $(n,m)$-diagram is an $(n,m)$-diagram which is either monic or epic.
  That is, a diagram is standard if it has maximal propagation number.
\end{definition}

The category $\TLcat$ is generated by standard diagrams and we have a ``Robinson-Schensted type'' correspondence.
\begin{lemma}\label{lem:diagram-tableaux}
  Let $r = (n-m)/2$ for $n\ge m$.  There is a bijection between $(n,m)$ standard diagrams and standard $(n-r,r)$ Young tableaux.  As such, the number of such diagrams is
  \begin{equation}
    \binom{n}{r} - \binom{n}{r-1}.
  \end{equation}
\end{lemma}
\begin{proof}
  Recall we label the sites from the top down.
  Consider those source sites that are connected to another source site above them.
  Clearly there are $r$ of these closing sites, and they uniquely determine the diagram.

  The bijection sends this diagram to the tableau with second row containing the labels of these closing sites.
  It is thus necessary only to show that the resultant tableau is standard as a counting argument (such as the hook-length formula for standard tableaux) shows the numerical result.

  Note that the top row of the tableau contains all the site labels of sites either connected by through wires or connected to source sites below them.
  In order for the diagram to be planar, we require that the $i$-th closing site has label at least $i$.
  However, for a two part partition, this is equivalent to the condition that the tableaux under the above bijection is standard.
\end{proof}

The category $\TLcat$ is naturally self-dual by the functor $\iota$ which sends a diagram $\underline{n} \to \underline{m}$ to its ``vertical reflection'' $\underline{m} \to \underline{n}$.
This descends to an antiautomorphism on $\TL_n$ also given by reflection about the vertical axis.
We call such a morphism an \emph{involution} and it equips all $\TL_n$-modules with a concept of a dual:
\begin{equation}
  M^* = \Hom_{\TL_n}(M, R)
\end{equation}
where the action on the homomorphism space is $x\cdot\phi(m) = \phi(\iota x \cdot m)$.

\begin{remark}
  Note that the regular $\TL_n$ module is \emph{not} necessarily self-dual.
  Indeed, we will later show that if $\delta =1$ and $R$ is any field, then $\TL_3$ decomposes as the direct sum of two modules: a self-dual, 3--dimensional projective cover of the trivial module and another module which is 2--dimensional and not self-dual.
\end{remark}

We introduce Dirac ``bra-ket'' notation for diagrams.
Let $x$ and $y$ be epic diagrams $\underline{n} \to \underline{m}$ where $n \ge m$.
We will then denote the corresponding morphisms in $\Hom_{\TLcat}(\underline{n}, \underline{m})$ by kets $|x\rangle$ and $|y\rangle$.
The image of these morphisms under $\iota$ are $\langle x|$ and $\langle y|$.
Thus we have $|x\rangle\langle y| \in \End_\TLcat(\underline{n})$ and $\langle x|y \rangle \in \End_\TLcat(\underline{m})$.
Care should be taken by those familiar with Dirac notation not to confuse the morphism $\langle - | - \rangle$ with the bilinear form $\langle -, - \rangle$ introduced later.

%

\section{Quantum Numbers}\label{sec:quantum_numbers}
\subsection{Introduction}
We briefly review the relevant theory of quantum numbers, or Chebyshev polynomials of the second kind.
In what follows in this section, $\delta$ is an indeterminate unless otherwise specified.
Alternatively, $\delta$ can be considered as an element of the pointed ring $(\Z[\delta],\delta)$.

\begin{definition}
  The quantum numbers, $[n]$ for $n\in \Z$, are polynomials in $\Z[\delta]$ that satisfy $[0] = 0$, $[1] = 1$ and $\delta[n] = [n+1] + [n-1]$.
\end{definition}

Note that the recurrence relations $[n\pm1] = \delta [n] - [n\mp1]$ shows that the polynomials are well defined for all $n\in \Z$ and indeed $[-n] = -[n]$.

The first few quantum numbers are thus
\vspace{1em}
\begin{center}
  \begin{tabular}{clccl}
    \toprule
    $n$      & $[n]$                &\hspace{3em}  & $n$ & $[n]$ \\
    \midrule
    0 & $0$                  &  & 5 & $\delta^4 - 3\delta^2 + 1$ \\
    1 & $1$                  &  & 6 & $\delta^5 - 4\delta^3 + 3\delta$ \\
    2 & $\delta$             &  & 7 & $\delta^6 - 5\delta^4 + 6\delta^2 - 1$ \\
    3 & $\delta^2 - 1$       &  & 8 & $\delta^7 - 6\delta^5 + 10\delta^3 - 4\delta$ \\
    4 & $\delta^3 - 2\delta$ &  & 9 & $\delta^8 - 7\delta^6 + 15\delta^4 -10\delta^2 + 1$ \\
    \bottomrule
  \end{tabular}
\end{center}
\vspace{1em}

The coefficients of the quantum numbers, $[n] = \sum_{i=0}^{n-1}c_{n,i}\delta^i$ obey the relation
\begin{equation}
c_{n,i} = c_{n-1,i-1} - c_{n-2,i}
\end{equation}
and form a ``half Pascal triangle''.  A closed form is given by
\begin{equation}
  c_{n, n-1-2i} = {(-1)}^i \binom{n-1-i}{i}.
\end{equation}

Quantum numbers are a ``$q$-analogue'' of the integers in the following way.
Consider the ring of symmetric Laurent polynomials in indeterminate $q$ over the integers.
This ring is isomorphic to $\Z[\delta]$ by $\delta \mapsto q+q^{-1}$.
Under this identification,
\begin{equation}\label{eq:q_quant_def}
  [n] = q^{n-1} + q^{n-3} + \cdots + q^{-n+3} + q^{-n+1} = \frac{q^n - q^{-n}}{q-q^{-1}}
\end{equation}
and we see that under the specification $q = 1$ (or equivalently under $\delta = [2] = 2$), quantum numbers ``specialise'' to normal numbers so that $[n] = n$.
Where necessary we will subscript $\delta$ or $q$ to specify the variable in use.
This ``$q$-formulation'' makes it clear that if $n \mid m$ then $[n] \mid [m]$.

Using the definition in \cref{eq:q_quant_def}, we can state and prove the following.
\begin{lemma}\label{lem:quantum_shatter}
  For any $m$ and $\ell$,
  ${[m\ell]}_q = {[\ell]}_q{[m]}_{q^\ell}$.
\end{lemma}
\begin{proof}
  This is a simple calculation in the subring of symmetric Laurent polynomials in $q$:
  \begin{equation}\label{eq:calcml}
    {[m\ell]}_q = \frac{q^{m\ell}-q^{-m\ell}}{q - q^{-1}}
    = \frac{q^{\ell} - q^{-\ell}}{q - q^{-1}}
      \frac{q^{m\ell} - q^{-m\ell}}{q^\ell - q^{-\ell}}
      = {[\ell]}_q{[m]}_{q^\ell}.
  \end{equation}
\end{proof}
Note that this is a relationship between polynomials and thus has a version in terms of $\delta$ notation:
\begin{equation}
  {[m\ell]}_\delta = {[\ell]}_\delta{[m]}_{[\ell+2]-[\ell]}
\end{equation}
Here the subscripts dictate the value of $\delta$ to use when evaluating the quantum number (not the value of $q$).
This is important as it is an equation in $\delta$ and makes no recourse to values such as $q^\ell - q^{-\ell}$ which may vanish under specialisation.

\begin{definition}
  The quantum factorial, $[n]!{}$ is defined for $n \in \N$ by $[0]! = 1$ and $[n]! = [1][2]\cdots[n]$.
\end{definition}

\subsection{Normalisation}
Readers familiar with quantum numbers may be accustomed to the normalisation
\begin{equation}
  \widetilde{[ n]}_q = 1 + q + \cdots + q^{n-1}.
\end{equation}
The two definitions are related by
\begin{equation}
  [n] = q^{-n+1}\widetilde{[n]}_{q^2}.
\end{equation}
Each has their advantage and place.
For example, $\widetilde{[n]}_q!{}$ counts the number of elements of $\mathfrak{S}_n$, stratified by inversion number, while the normalisation $[n]$ will be crucial to the sequel.

\subsection{Specialisation}
Let us now fix in mind a domain $R$ and some element of $R$, which we will also call $\delta$.
The object $(\Z[\delta],\delta)$ is an initial object in the category of pointed rings, so we may consider the canonical images of elements of $\Z[\delta]$ in the pointed ring $(R,\delta)$.
The context should make it clear if we refer to an element of $R$ or $\Z[\delta]$.
Thus we may consider $[n]$ as an element of $R$ which is polynomial in $\delta \in R$.
This highlights a reason for our formulation in terms of $\delta$ as opposed to $q$: we make no assumption that any elements in $R$ have inverses.

We are now interested in the vanishing of $[n]$.
Suppose $[n]$ does vanish for some $n \neq 0$.
Since $[n] \mid [m]$ whenever $n \mid m$, there must be a least positive $\ell$ such that $[\ell] = 0$ and all zeros occur at multiples of $\ell$.
If $[n]$ never vanishes in $R$, we will take $\ell = \infty$.
This is known as the \emph{quantum characteristic} of the pair $(R,\delta)$.

A simple inductive proof shows that $[\ell - n] = -[\ell + n]$ and in particular, $[2\ell -1] = -1$.
Further, $[\ell - k] = [\ell - 1][k]$, and so $[\ell - 1] = \pm 1$.

Specialising to a ring $R$ and element $\delta$ introduces two ``twists'' as we have seen.
The natural number $\ell$ is the ``quantum-torsion'' of this situation, and there is a smallest natural $p$ such that $p\cdot 1 = 0$ in $R$.
If $R$ is a domain (as will often be the case), $p$ is prime.
We say that the situation over such a pointed ring is
\begin{itemize}
  \item ``semisimple'' if $\ell = \infty$
  \item ``characteristic zero'' if $\ell < \infty$ but $p = \infty$
  \item ``positive characteristic'' if $\delta = 2$ and so $\ell = p < \infty$ and $[n] = n$
  \item ``mixed characteristic'' in other cases
\end{itemize}

Of considerable importance will be the $(\ell, p)$-digits of natural numbers.
When an pointed ring $R$ is in mind, the $(\ell, p)$-digits of a number $n\in\N$ are those naturals $n_0, n_1, \cdots$ such that
\begin{equation}\label{eq:l_p_digits}
  n = n_0 + n_1 \ell + n_2 \ell p + n_3 \ell p^2 + \cdots + n_k \ell p^{k-1}
\end{equation}
for $0\le n_0 < \ell$ and $0 \le n_i < p$ for $i > 0$.
If we write $\ell p^{i-1} = p^{(i)}$ with the understanding that $p^{(0)} = 1$, then we may write $n = \sum_{i=0}^k n_i p^{(i)}$.
We may also simply write out the digits as $n = [n_k, n_{k-1}, \ldots, n_0]_{p,\ell}$.

\subsection{Difference Series}\label{sec:quantum:difference}
Let $R$ be a domain.
We now consider the series given by $\Delta_0 = 2$, $\Delta_1 = \delta$ and obeying the same recurrence relation as the quantum numbers.
This series arises as $\Delta_n = [n+1] - [n-1]$.

If $\ell = 2\ell'$ is even and $[\ell-1] = -1$, then $2[\ell'] = 0$ (as $[2\ell' - k] = -[k]$) and $[2][\ell'] =[\ell'+1] + [\ell'-1] = 0$.
By the minimality of $\ell$, both 2 and $[2]$ vanish (so $\ell = 2$).
In general, if we are suffering from 2-torsion, then $\Delta_n = \delta[n]$.
On the other hand, if $\delta = 0$, then
\begin{equation*}
  [n] = \begin{cases}
    {(-1)}^k & n = 2k-1\\
    0 & \text{else}
  \end{cases}\quad\quad\quad\quad\quad
  \Delta_n =  \begin{cases}
    2{(-1)}^k & n = 2k\\
    0 & \text{else}
  \end{cases}
\end{equation*}

The remaining cases to examine are $\ell$ even with $[\ell-1] = 1$ and $\ell$ odd.

\begin{lemma}\label{lem:quantum_far_diff}
  For all integers $n$ and $k$, $[n+k] - [n-k] = [k]\Delta_{n}$.
\end{lemma}
\begin{proof}
  Clearly it suffices to show the result for non-negative $k$.
  It is then a simple induction on $k$, with the case $k=0$ being trivial and the inductive step
  \begin{align*}
    [n+k+1] - [n-k-1] &= \delta[n+k] - [n+k-1] - \delta [n-k] + [n-k+1]\\
    &= \delta[k]\Delta_{n} - [k-1]\Delta_{n}\\
    &= [k+1]\Delta_{n}.
  \end{align*}
\end{proof}

\begin{lemma}
  If $\delta = \pm 2$ then $\Delta_n = 2{(\pm1)}^n$ for all $n$.
  Otherwise for $n \equiv_2 m$, we have that $\Delta_n = \Delta_m$, if and only if either $n \equiv_{2\ell} m$ or $n \equiv_{2\ell} -m$.
\end{lemma}
\begin{proof}
  The first statement is clear.  Thus we may assume $\Delta \neq \pm 2$.

  Suppose $n>m$ and write $a = (n+m)/2$ and $b = (n-m)/2$.   Then
  \begin{align*}
    \Delta_n - \Delta_m &= [n + 1] - [m+1] - ([n-1] - [m-1])\\
    &= [a+b+1] - [a-b+1] - ([a+b-1] - [a-b-1])\\
    &= [b](\Delta_{a+1} - \Delta_{a-1})
  \end{align*}
  where the last equality follows from \cref{lem:quantum_far_diff}.

  Since $R$ is a domain, we have that if $\Delta_n = \Delta_m$, either $\ell | b$ or $2[a] = [a+2] + [a-2]$.
  However,
  \begin{align*}
    (\delta^2 - 2)[a] &= \delta([a+1] + [a-1]) - 2 [a]\\
    &= [a+2] + [a-2]
  \end{align*}
  and so in the second case, $(\delta^2 - 4)[a] = 0$.
  Since $\delta \neq \pm 2$, the quantum number $[a]$ must vanish so $\ell \mid a$.
  We have thus shown that $\ell \mid (n+m)/2$ or $\ell \mid (n-m)/2$ as desired.

  Recall $[2\ell -1]= -1$ so that $[2\ell + t] = [t]$.  That is to say that $[n]$ (and hence $\Delta_n$) is periodic with period $2\ell$.
  Further $\Delta_{-n} = [-n+1]-[-n-1] = [n+1]-[n-1] = \Delta_n$.
  This shows the converse.
\end{proof}

The principal application of this lemma is as follows.
Suppose $\delta \neq \pm 2$.
If we consider the integer number line and let the infinite dihedral group $D_\infty = \langle \rs, \rt : e = \rs^2 = \rt^2\rangle$ act by reflection about numbers one less than a multiple of $\ell$ (so that $\rs$ is reflection about $-1$ and $\rt$ is reflection about $\ell-1$), then the value of $\Delta_n$ describes exactly the orbits of this action.  This is shown in \cref{fig:orbits} for $\ell = 5$.

\begin{figure}[htpb]
  \centering
  \begin{tikzpicture}[scale=0.5]
  \draw (-1,0) -- (26, 0);
  \draw (4, -1) -- (4,1);
  \draw (9, -1) -- (9,1);
  \draw (14, -1) -- (14,1);
  \draw (19, -1) -- (19,1);
  \draw (24, -1) -- (24,1);
  \foreach \i in {0,...,25}{
    \draw[line=white,fill=white] (\i, -0.5) circle (0.2);
    \node at (\i, -0.5) {\small$\i$};
  }
  \draw[thick,<->,color=red] (4.65,0.5) arc (-20:200:0.7);
  \draw[thick,<->,color=blue] (9.65,0.5) arc (-20:200:0.7);
  \draw[thick,<->,color=red] (14.65,0.5) arc (-20:200:0.7);
  \draw[thick,<->,color=blue] (19.65,0.5) arc (-20:200:0.7);
  \draw[thick,<->,color=red] (24.65,0.5) arc (-20:200:0.7);

  \draw[fill, orange] (0, 0) circle (0.1);
  \draw[fill, orange] (8, 0) circle (0.1);
  \draw[fill, orange] (10, 0) circle (0.1);
  \draw[fill, orange] (18, 0) circle (0.1);
  \draw[fill, orange] (20, 0) circle (0.1);

  \draw[fill, green] (1, 0) circle (0.1);
  \draw[fill, green] (7, 0) circle (0.1);
  \draw[fill, green] (11, 0) circle (0.1);
  \draw[fill, green] (17, 0) circle (0.1);
  \draw[fill, green] (21, 0) circle (0.1);

  \draw[fill, purple] (2, 0) circle (0.1);
  \draw[fill, purple] (6, 0) circle (0.1);
  \draw[fill, purple] (12, 0) circle (0.1);
  \draw[fill, purple] (16, 0) circle (0.1);
  \draw[fill, purple] (22, 0) circle (0.1);

  \draw[fill, cyan] (3, 0) circle (0.1);
  \draw[fill, cyan] (5, 0) circle (0.1);
  \draw[fill, cyan] (13, 0) circle (0.1);
  \draw[fill, cyan] (15, 0) circle (0.1);
  \draw[fill, cyan] (23, 0) circle (0.1);
  \draw[fill, cyan] (25, 0) circle (0.1);

  \draw[fill, red] (4, 0) circle (0.1);
  \draw[fill, red] (14, 0) circle (0.1);
  \draw[fill, red] (24, 0) circle (0.1);
  \draw[fill, blue] (9, 0) circle (0.1);
  \draw[fill, blue] (19, 0) circle (0.1);
\end{tikzpicture}
  \caption{The orbits of $D_\infty$ on the natural number line for $\ell = 5$.
  Note that these are described exactly by equivalence classes given by the values of $\Delta_n$ where $[5] = 0$.}%
  \label{fig:orbits}
\end{figure}
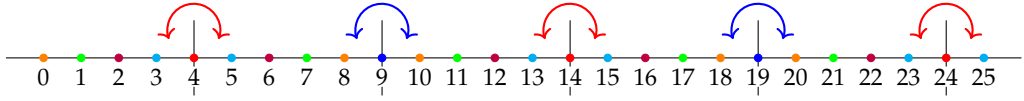

\subsection{Combinatorics}\label{subsec:comb}
A frequently used quantum analogue is the Gaussian binomial coefficient.
\begin{definition}
  The $(n,r)$-th Gaussian binomial coefficient is given by
  \begin{equation}
    \gaussianquant{n}{r} = \frac{[n]!}{[r]![n-r]!}.
  \end{equation}
\end{definition}
The moderately remarkable fact is that such coefficients, although expressed as a rational function are actually polynomials in $\delta$.
In the $\widetilde{[n]}_q$ formulation, they can be derived as the coefficient of $X^r Y^{n-r}$ in ${(X+Y)}^n$ in $\Z[q,X,Y] / (XY - qYX)$.

There is a form of Lucas' Theorem for specialised quantum binomials.  It is as follows
\begin{theorem}\label{thm:lucas}
  Let $n = [n_k, n_{k-1},\ldots, n_0]_{p,\ell}$ and $r = [r_k, r_{k-1},\ldots, r_0]_{p,\ell}$.
  Then
  \begin{equation}
    \gaussianquant{n}{r} =
    \binom{n_k}{r_k}\cdots \binom{n_1}{r_1}.
    \gaussianquant{n_0}{r_0}.
  \end{equation}
\end{theorem}
The pattern of vanishing quantum binomials is thus beautifully fractal in nature, foreshadowing the fractal-like answer to the principal questions in this paper.
This is illustrated by \cref{fig:sierpinski}.
Let us write $n \triangleright r$ if $n_i \ge r_i$ for all $i$.

\begin{figure}[htpb]
\begin{center}
  \includegraphics[width=0.5\textwidth]{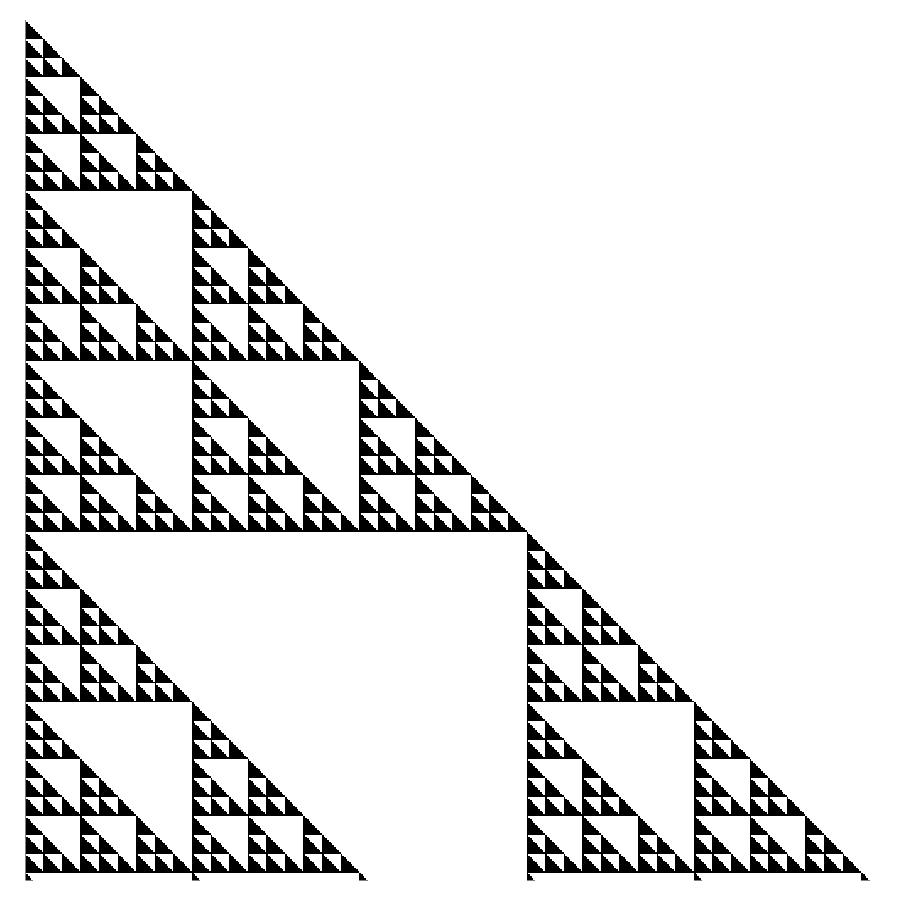}
\end{center}
\caption{A diagram showing the vanishing of quantum binomials for $p=3$ and $\ell = 11$.  Note the ``generalised Sierpinski'' form.}%
\label{fig:sierpinski}
\end{figure}

\begin{corollary}\label{cor:nonzero_binoms}
  The quantum binomial $\gaussianquant{n}{r}$ is non-zero iff $n \triangleright r$.
  Hence $\gaussianquant{n}{r}$ is non-zero for all $0 \le r \le n$ iff $n$ is less than $\ell$ or of the form $a p^{(k)}-1$ for some $0< k$ and $0< a < p$.
\end{corollary}

We recount a proposition that is almost folklore.
Recall that a (rooted) forest is a partially ordered set such that if $a \le b$ and $a \le c$ then either $b \le c$ or $c \le b$.
\begin{proposition}\cite{stanley_1972}\label{prop:h_F}
  Let $(F, \le)$ be a forest and $c : F \to \N$ the function $c(x) = |\{y \in F \;:\; y\le x\}|$.
  Then the ratio
  \begin{equation}
    h_F = \frac{\big[|F|\big]!}{\Pi_{x\in F}[c(x)]}
  \end{equation}
  is a polynomial in $\delta$.
\end{proposition}
We recount a proof for completeness.
\begin{proof}
  We induct on the cardinality of $F$.
  If $F$ is a singleton, the result clearly holds.

  If $F$ is a true forest (i.e.\ not a single tree) then suppose $F = F_1 \oplus F_2$ is a disjoint sum of forests.
  By induction $h_{F_1}$ and $h_{F_2}$ are polynomials in $\delta$.
  But 
  \begin{equation}
    h_F = \frac{\big[|F|\big]!}{\Pi_{x\in F}[c(x)]}
        = \frac{\big[|F|\big]!}{\big[|F_1|\big]!\big[|F_2|\big]!}\frac{\big[|F_1|\big]!}{\Pi_{x\in F_1}[c(x)]}\frac{\big[|F_2|\big]!}{\Pi_{x\in F_2}[c(x)]}
        = 
        \gaussianquant{|F|}{|F_1|}h_{F_1}h_{F_2}
  \end{equation}
  and so the result holds for $F$ by the fact that Gaussian binomial coefficients are polynomials.

  On the other hand, if $F$ is a tree and $r$ is it's root, with subforest $F \setminus \{r\} = F'$,
  \begin{equation}
    h_F = \frac{\big[|F|\big]!}{\Pi_{x\in F}[c(x)]}
        = \frac{\big[|F|\big]!}{\big[|F|\big] \Pi_{x\in F'}[c(x)]}
        = \frac{\big[|F|-1\big]!}{\Pi_{x\in F'}[c(x)]}
        = h_{F'}.
  \end{equation}
\end{proof}
The proof above makes it clear that the non-quantum version (setting $\delta = 2$ in $\Z$) reduces to $h_F$ counting the number of distinct order preserving maps $F \to \{1, \ldots, |F|\}$.
This is a form of ``hook length formula'' for forests.

\begin{definition}\label{def:h_F_x}
  For a monic $(n,m)$-diagram $x$, we can form an monic $(n+m,0)$-diagram by ``rotating'' the $m$ right hand points ``downwards'' to the left hand side.
  Formally, this is a diagram where source site $i$ is connected to source site $j$ if $1\le i,j \le n$ and source site $i$ is connected to source site $j$ in $x$ or if $1\le i\le n < j \le n+m$ and source site $i$ is connected to target site $m-j+n+1$ in $x$.
  We form the forest $F(x)$ of $x$ whose elements are edges in the $(n+m,0)$-diagram, and where $e_1 \le e_2$ iff $e_1$ is contained within the area described by $e_2$ and the left border of the diagram.
\end{definition}

The value $h_{F(x)}$ will be important in defining morphisms between modules of $\TL_n$ and also arrises in a form for Jones-Wenzl idempotents.

\section{Cellular Algebras}\label{sec:cellular}
We recount here the basic theory of cellular algebras as introduced by Graham and Lehrer.
The algebra $\TL^R_n(\delta)$ is cellular for all choices of $R$, $n$ and $\delta$, and may even be considered the canonical example.

\begin{definition}\label{def:cellular_algebra}\cite[1.1]{graham_lehrer_1996}
  An algebra $A$ over ring $R$ is termed cellular if there is a tuple of ``cell data'' $(\Lambda, M, C, \iota)$ where
  \begin{enumerate}
    \item the set $\Lambda$ of ``cell indices'' is partially ordered,
    \item for each $\lambda\in \Lambda$, $M(\lambda)$ is a finite set of ``$\lambda$-tableau'',
    \item the map $C : \amalg_{\lambda \in \Lambda} M(\lambda)\times M(\lambda) \to A$ sending $(m_1, m_2) \mapsto C^\lambda_{m_1,m_2}$ is injective and describes a $R$-basis for $A$, and
    \item the involution $\iota : A \to A$ sends $C^\lambda_{m_1, m_2}$ to $C^\lambda_{m_2, m_1}$.
  \end{enumerate}
  Finally, let $A^{<\lambda}$ denote the $R$-span of all $C^\mu_{m_1, m_2}$ for $\mu < \lambda$ and similarly for $A^{\le \lambda}$.
  Then we require that for all $m_1, m_2 \in M(\lambda)$ and $a\in A$, there is a form $\langle -, - \rangle_a$ on $M(\lambda)$ satisfying
  \begin{equation}\label{eq:cellular_def}
    a \cdot C^\lambda_{m_1, m_2}
    =
    \sum_{m_3\in M(\lambda)}\langle m_3, m_1\rangle_a C^\lambda_{m_3, m_2}
    \quad\quad\mod A^{< \lambda}.
  \end{equation}
\end{definition}
It is clear that \cref{eq:cellular_def} implies that $A^{\le \lambda}$ is an ideal for all $\lambda \in \Lambda$ and so $A^{< \lambda}$ is too.

Cellular algebras are endowed with two important cellular structures: cell ideal chains and cell modules.
Let $A$ be a cellular algebra with data $(\Lambda, M, C,\iota)$.
If $(\lambda_1, \ldots, \lambda_m)$ is an ordering of the cell indices such that $\lambda_i \le \lambda_j$ implies $i\le j$, then there is a chain of ideals
\begin{equation}
  0 = J_0 \subset J_1 \subset \cdots \subset J_m = A
\end{equation}
such that $J_i/J_{i-1}$ is spanned by ${\{C^{\lambda_i}_{m_1, m_2} + J_{i-1}\}}_{m_1,m_2\in M(\lambda_i)}$.
If $\Lambda$ is totally ordered, then the only cell ideal chain is
\begin{equation}
  0 = A^{<\lambda_1} \subset A^{<\lambda_2} \subset \cdots \subset A^{<\lambda_m} \subset A.
\end{equation}

The second important structure is the cell module indexed by $\lambda$. It is denoted $W(\lambda)$.
It is defined to be the $R$-span of the elements of $M(\lambda)$ with $A$-action given by
\begin{equation}
  a \cdot m = \sum_{m'\in M(\lambda)}\langle m', m \rangle_a m'.
\end{equation}

Substantial clarity can be obtained with the following observation.
By applying $\iota$ to \cref{eq:cellular_def}
\begin{equation}
  C^\lambda_{m_1, m_2}\cdot \iota(a)
  =
  \sum_{m_3\in M(\lambda)}\langle m_3, m_2\rangle_a C^\lambda_{m_1, m_3}
\quad\quad\mod A^{<\lambda}.
\end{equation}
A direct result is the following.
\begin{lemma}\cite[1.7]{graham_lehrer_1996}
  For any $a \in A$, $\lambda\in \Lambda$ and $m_1, m_2, m_3,m_4 \in M(\lambda)$,
  \begin{equation}
    C^\lambda_{m_1,m_2} a C^\lambda_{m_3,m_4}
    =
    \phi_a(m_2,m_3)
    C^\lambda_{m_1, m_4}
    \quad\quad \mod A^{<\lambda}
  \end{equation}
  for some map $\phi_a : M(\lambda)\times M(\lambda)\to R$.
\end{lemma}
This defines a bilinear form on $W(\lambda)$ by extending $\phi_1$ from the basis $M(\lambda)$ to the entire space.
We denote this form $\langle - , - \rangle_\lambda$.
When clear, we will omit the subscript.

These forms dictate the structure of the algebra $A^{\le \lambda}/A^{<\lambda}$ in its entirety as
\begin{equation}
  C^\lambda _{m_1, m_2}
  C^\lambda _{m_3, m_4}
  = \langle m_2, m_3 \rangle C^\lambda_{m_1, m_4}
  \quad\quad\mod A^{<\lambda}.
\end{equation}

Further, these forms regulate the representation theory of $A$.
Henceforth assume $R$ is a field.
\begin{theorem}\label{thm:cell_reps}\cite[3.2,3.4,3.8]{graham_lehrer_1996}
  Let $R(\lambda)\subseteq W(\lambda)$ be the kernel of the bilinear form $\langle -,- \rangle_\lambda$.
  Further, let $\Lambda_0=\{\lambda\in\Lambda : \langle -,- \rangle \neq 0\}$.
  Then
  \begin{enumerate}
    \item $A$ is semisimple iff $R(\lambda) = 0$ for all $\lambda\in\Lambda$,
    \item if $\lambda\in\Lambda_0$, then the head of $W(\lambda)$ is $L(\lambda) = W(\lambda)/R(\lambda)$ and is absolutely irreducible, and
    \item ${\{W(\lambda)/R(\lambda)\}}_{\lambda\in\Lambda_0}$ is a complete set of non-isomorphic simple modules for $A$.
  \end{enumerate}
\end{theorem}

Note that $A^{\le \lambda}/A^{\lambda}$ is isomorphic to the direct sum of $|M(\lambda)|$ copies of $W(\lambda)$ as a left module over $A^{\le \lambda}$.
Indeed for each $m_2 \in M(\lambda)$, the $R$-span of $\{C^{\lambda}_{m_1,m_2} + A^{< \lambda}: m_1 \in M(\lambda)\}$ is a subspace of $A^{\le \lambda}/ A^{< \lambda}$ invariant under action of $A^{\le\lambda}$ and isomorphic as an $A$-module to $W(\lambda)$.
Further these modules all intersect trivially.

Let $d_{\lambda,\mu}$ be the composition multiplicity $[W(\lambda):L(\mu)]$.
The decomposition matrix $D = (d_{\mu,\lambda})$ is the primary object of study.
If the projective cover of $L(\lambda)$ is denoted $P(\lambda)$ so that $c_{\lambda,\mu} = [P(\lambda): L(\mu)]$ gives the Cartan matrix $C = (c_{\lambda,\mu})$, then we have the following crucial result.
\begin{theorem}\cite[3.7]{graham_lehrer_1996}\label{thm:gg}
  With respect to the ordering on $\Lambda$,
  the decomposition matrix is upper uni-triangular and $C = D^{t}D$.  That is,
  \begin{equation}
    [P(\lambda): L(\mu)] = \sum_{\nu} [W(\lambda) : L(\nu)] [W(\mu) : L(\nu)]
  \end{equation}
\end{theorem}

As mentioned above, the Temperley--Lieb algebra is in many ways the quintessential cellular algebra --- in fact it is the third example in the paper defining cellular algebras.
\begin{proposition}\cite[6.7]{graham_lehrer_1996}
  For arbitrary commutative pointed ring $(R,\delta)$, the Temperley--Lieb algebra $\TL^R_n(\delta)$ is cellular with cell data:
  \begin{enumerate}
    \item cell indices $\Lambda = \{ m : 0\le m \le n \mbox{ and } m\equiv_2 n\}$ inheriting the usual order,
    \item $m$-tableaux monic diagrams morphisms $\underline{n} \to \underline{m}$,
    \item basis $C^m_{x,y}$ given by $|x\rangle\langle y|$ for monic diagrams $x,y :\underline{n} \to \underline{m}$, and
    \item the anti-automorphism $\iota$.
  \end{enumerate}
  The standard form is defined on the basis so that
  \begin{equation}
    \langle x | y \rangle  = \langle x, y\rangle_{m}\id_{m} + A^{<m}.
  \end{equation}
\end{proposition}

It will be shown in \cref{lem:form_nondegenerate} that $\Lambda_0 = \Lambda$ unless $n$ is even and $\delta = 0$ in which case $\Lambda_0 = \Lambda \setminus \{0\}$.

\section{Cell Modules}\label{sec:cell-modules}
We introduce the cell modules as defined by the cellular structure of $\TL_n$ in a diagrammatic manner.
To keep parity with the representation of other finite-dimensional algebras, we will label the module $W(m)$ for $\TL_n$ as $S(n,m)$ and may use the term ``standard module'' interchangeably.

\begin{definition}
  For arbitrary $R$ and $\delta \in R$, let $M(n,m)$ be the homomorphism space $\Hom_{\TLcat}(\underline{n}, \underline{m})$ with natural (left) $\TL_n$ action.
  Then $S(n,m)$ is defined to be the quotient of $M(n,m)$ by the submodule of all morphisms factoring through $\underline{i}$ for any $i<m$.
\end{definition}
This slightly opaque definition has a clear diagrammatic presentation.
Firstly, note that $S(n,m) = 0$ unless $m\le n$ and is of the same parity.
The module $S(n,m)$ has basis given by monic $(n,m)$-diagrams.
The action of $\TL_n$ is given by standard diagrammatic composition, but any resulting diagram that is not monic is killed.
An example of $u_4$ acting on an element of $S(5,3)$ is given below:

\begin{center}
    $
  \vcenter{\hbox{
  \begin{tikzpicture}[scale=0.5]
    \draw (2.0, -.5) -- (2.0,4.5);
    \draw (0.0, -.5) -- (0.0,4.5);
    \foreach \i in {0,...,4} {
      \draw[fill] (2, \i) circle (0.1);
      \draw[fill] (0, \i) circle (0.1);
    }
    \draw[very thick] (2, 4) to[out=180, in=0] (0,4);
    \draw[very thick] (2, 2) to[out=180, in=0] (0,2);
    \draw[very thick] (2, 3) to[out=180, in=0] (0,3);

    \draw[very thick] (0, 0) to[out=0, in=0] (0,1);
    \draw[very thick] (2, 0) to[out=180, in=180] (2,1);
  \end{tikzpicture}
}} \quad\cdot\quad
  \left(
  \vcenter{\hbox{
  \begin{tikzpicture}[scale=0.5]
    \draw (2.0, 0.5) -- (2.0,3.5);
    \draw (0.0, -.5) -- (0.0,4.5);
    \foreach \i in {0,...,2} {
      \draw[fill] (2, \i+1) circle (0.1);
    }
    \foreach \i in {0,...,4} {
      \draw[fill] (0, \i) circle (0.1);
    }
    \draw[very thick] (2, 1) to[out=180, in=0] (0,0);
    \draw[very thick] (2, 2) to[out=180, in=0] (0,3);
    \draw[very thick] (2, 3) to[out=180, in=0] (0,4);

    \draw[very thick] (0, 1) to[out=0, in=0] (0,2);
  \end{tikzpicture}
}} \;+ 
  \vcenter{\hbox{
  \begin{tikzpicture}[scale=0.5]
    \draw (2.0, 0.5) -- (2.0,3.5);
    \draw (0.0, -.5) -- (0.0,4.5);
    \foreach \i in {0,...,2} {
      \draw[fill] (2, \i+1) circle (0.1);
    }
    \foreach \i in {0,...,4} {
      \draw[fill] (0, \i) circle (0.1);
    }
    \draw[very thick] (2, 1) to[out=180, in=0] (0,0);
    \draw[very thick] (2, 2) to[out=180, in=0] (0,1);
    \draw[very thick] (2, 3) to[out=180, in=0] (0,4);

    \draw[very thick] (0, 2) to[out=0, in=0] (0,3);
  \end{tikzpicture}
}}
  \right)
   \quad= \quad
  \vcenter{\hbox{
  \begin{tikzpicture}[scale=0.5]
    \draw (2.0, 0.5) -- (2.0,3.5);
    \draw (0.0, -.5) -- (0.0,4.5);
    \foreach \i in {0,...,2} {
      \draw[fill] (2, \i+1) circle (0.1);
    }
    \foreach \i in {0,...,4} {
      \draw[fill] (0, \i) circle (0.1);
    }
    \draw[very thick] (2, 1) to[out=180, in=0] (0,2);
    \draw[very thick] (2, 2) to[out=180, in=0] (0,3);
    \draw[very thick] (2, 3) to[out=180, in=0] (0,4);

    \draw[very thick] (0, 0) to[out=0, in=0] (0,1);
  \end{tikzpicture}
}}
  $
\end{center}
\vspace{1em}

\begin{remark}
  In~\cite{graham_lehrer_1998}, Graham and Lehrer embrace the category theory and construct the cell modules as follows.
  Any functor $F : \TLcat^R \to \cmod{R}$ describes a family of modules for the Temperley--Lieb algebras naturally.
  Indeed, the module ``at'' $\TL_n^R$ is given by $F(\underline{n})$ as an $R$-space with action given by $a \cdot v = F(a)v$ for any $v\in F(\underline{n})$ and $a \in \TL_n^R = \End_{\TLcat}(\underline{n})$.
  This framework carries an advantage in that it recognises the place of $\TL_n$ within a broader structure.
  More specifically, morphisms in $\Hom_{\mathcal{TL}}(\underline{n}, \underline{m})$ can act on these modules.

  In this framework, morphisms between functors are restricted to such actions.
  To broaden into the general case where morphisms may be more arbitrary brings us to the representations of 2-categories which is outside the scope of this paper.

  However, the fact is that the only morphisms between cell modules do arise from elements of $\Hom_{\TLcat}(\underline{n},\underline{m})$ and so not much is lost in this methodology.
  This (along with the results in \cref{sec:truncation}) elucidate the ambivalence of the theory towards the parameter $n$: a result that is not obvious from the initial statement.
\end{remark}

As evidenced in \cref{sec:cellular}, a crucial role is played by the standard form on $S(n,m)$.
This is defined on diagrams $x$ and $y$ by the coefficient of $id_m$ in $\langle x | y \rangle$.
That is to say,
\begin{equation}\label{eq:form}
  \langle x | y \rangle = \langle x, y\rangle \id_m \quad \quad \mod \End_{\TLcat}^{<m}(\underline{m})
\end{equation}
\begin{lemma}\label{lem:form_nondegenerate}
  Over a pointed field, the standard form is non-degenerate iff $\delta$ is invertible or $m > 0$.
\end{lemma}
\begin{proof}
  If $\delta \neq 0$ then any $\langle x, x\rangle = \delta^{\frac{n-m}{2}} \neq 0$.
  Otherwise, consider the maps $z = \id_1\otimes \cap^{\otimes\frac{n-m}{2}} \otimes \id_{\frac{n+m}{2}-1}$ and $y = \cap^{\otimes\frac{n-m}{2}} \otimes \id_{\frac{n+m}{2}}$.
  That is,
  \begin{equation}
  z =
\vcenter{\hbox{
  \begin{tikzpicture}[scale=0.5]
    \draw (2.0, 0.5) -- (2.0,3.5);
    \draw (0.0, -.5) -- (0.0,5.5);

    \draw[very thick] (0, 0) to[out=0, in=180] (2, 1);
    \draw[very thick] (0, 0.5) to[out=0, in=0] (0, 1);
    \node at (0.2, 2.0) {$\vdots$};
    \draw[very thick] (0, 2.5) to[out=0, in=0] (0, 3);
    \draw[very thick] (0, 3.5) to[out=0, in=180] (2, 1.5);
    \node at (0.2, 4.3) {$\vdots$};
    \node at (1.8, 2.4) {$\vdots$};
    \draw[very thick] (0, 5.0) to[out=0, in=180] (2, 3.0);
  \end{tikzpicture}
}}\quad\quad\quad\quad\quad\quad
  y =
\vcenter{\hbox{
  \begin{tikzpicture}[scale=0.5]
    \draw (2.0, 1.0) -- (2.0,4.0);
    \draw (0.0, 0) -- (0.0,6.0);

    \draw[very thick] (0, 0.5) to[out=0, in=0] (0, 1);
    \node at (0.2, 2.0) {$\vdots$};
    \draw[very thick] (0, 2.5) to[out=0, in=0] (0, 3);
    \draw[very thick] (0, 3.5) to[out=0, in=180] (2, 1.5);
    \node at (0.2, 4.3) {$\vdots$};
    \node at (1.8, 2.4) {$\vdots$};
    \draw[very thick] (0, 5.0) to[out=0, in=180] (2, 3.0);
    \draw[very thick] (0, 5.5) to[out=0, in=180] (2, 3.5);
  \end{tikzpicture}
}}.
  \end{equation}
  Then clearly $\langle z|y\rangle = \id_m$ so $\langle z, y \rangle = 1$.

  On the other hand, if $\delta = 0$ and $m=0$, then any $\langle x | y \rangle$ for $n>0$ will be some power of $\delta$ multiplied by $\id_0$ and hence vanish.
\end{proof}
In fact the construction in the proof of \cref{lem:form_nondegenerate} is important because $\langle z | y \rangle = \id_m$ identically (and not only up to monic diagrams).

The kernel of the inner product is given by $R(n,m)$ and the irreducible quotient is $D(n,m)$.
As such, $\dim D(n,m) = \rk G(n,m)$ where $G(n,m)$ is the Gram matrix of $\langle -, - \rangle$.

\begin{example}
  Consider the module $S(n, n-2)$.
  It has a basis of $n-1$ diagrams $(x_i)_{i=1}^{n-1}$ where $x_i$ has a single simple link at $i$.
  The Gram matrix is
  \begin{equation}\label{eq:gram_n_2}
    G(n,n-2) =
    \begin{pmatrix}
      \delta & 1 & 0 & \cdots & 0 \\
      1 & \delta & 1 & \cdots & 0 \\
      0 & 1 & \delta & \cdots & 0 \\
      \vdots & \vdots & \vdots & \ddots & \vdots \\
      0 & 0 & \cdots & 1 & \delta \\
    \end{pmatrix}
  \end{equation}
  An easy induction shows that the determinant of such a $(n-1)\times (n-1)$ matrix is $[n]$.
  Thus when $\ell \nmid n$, the module $S(n, n-2)$ is irreducible.
  When it is not irreducible, the radical $R(n, n-2)$ is one dimensional.
  It is spanned by the element
  \begin{equation*}
    \sum_{i = 1}^{n-1}(-1)^{i+1}[i]x_i,
  \end{equation*}
  as can be calculated from the kernel of \cref{eq:gram_n_2}.
\end{example}

Substantial work has gone into the study of the determinant of the matrix $G(n,r)$.
The principle result, which holds over arbitrary characteristic, is a closed formula for the determinant.
\begin{proposition}\cite{ridout_saint_aubin_2014,westbury_1995}
  For any $R$ and $\delta$, if $r = (n-m)/2$,
  \begin{equation}
    \det G(n,m) = \prod_{j=1}^r {\left(\frac{[m+1+j]}{[j]}\right)}^{\dim S(n, m+ 2j)}
  \end{equation}
\end{proposition}
Note that despite being expressed as a rational function $\det G(n,m)$ is the determinant of a matrix with entries in $\{0,1,\delta, \delta^2,\ldots\}$, so the result is a polynomial in $\delta$.
Indeed it is this fact that allows us to claim that the results of~\cite{ridout_saint_aubin_2014,westbury_1995} hold over arbitrary ring.
This is an example of a principle similar to the ``evaluation principle'' set out by Goodman and Wenzl~\cite{goodman_wenzl_1993}.

When studying the representation theory of $\TL_n$ over characteristic zero, a key result is the following.
It is a general feature of the algebra $\TL_n$ over any ring that the structure of $S(n,m)$ when $m\equiv_{\ell} -1$ is the most difficult to ascertain.
\begin{theorem}
  If $(R,\delta)$ is characteristic zero, then when $m\equiv_{\ell} - 1$, the form on $S(n,m)$ is non-degenerate and $D(n,m) \simeq S(n,m)$.
\end{theorem}
\begin{proof}
  Suppose $m + 1 = m' \ell$ so
  \begin{equation}
    \det G(n,m) = \prod_{j = 1}^r{\left(\frac{[m'\ell + j]}{[j]}\right)}^{\dim S(n,m+2j)}.
  \end{equation}
  All factors where $\ell \nmid j$ are non-zero.
  When $j = j'\ell$, the fraction is
  \begin{equation}
    \frac{[(m' + j')\ell]}{[j'\ell]} = \frac{{[m'+j']}_{\pm 2}}{{[j']}_{\pm 2}} = (\pm 1)^{m'}\frac{m'+j'}{j'}
  \end{equation}
  by \cref{lem:quantum_shatter}.
  But $j'$ is invertible and $m' + j'$ is non-zero as $R$ is characteristic zero, for all $1 \le j' \le (n-m)/(2\ell)$, so these fractions are non-zero too.
\end{proof}
As a result, the representation theory of $\TL_n$ is well understood in the characteristic zero case and can be found in~\cite{ridout_saint_aubin_2014}.
However, we will see that this is a special case of the following analysis, if we set $p = \infty$.

Our focus is in the modular representation theory and for this we note the following.  Here, let $\nu_p$ be the $p$-adic valuation on the rationals, extended to $\Z[\delta] / [\ell]$.
Note that the above theorem shows that if $m\equiv_\ell -1$ that $\det G(n,m)$ is not divisible by $[\ell]$ and so does not vanish in $\Z[\delta] / [\ell]$.
\begin{theorem}
  Suppose $R$ has characteristic 0, and $n, m\equiv_{\ell} -1$.  If $\lfloor(n - m ) / 2\ell\rfloor < p^{\nu_p(m+1)}$ then $\nu_p(\det G(n,m)) = 0$.
\end{theorem}
\begin{proof}
  As above, let $m+ 1 = m' \ell$, and $n+1 = n' \ell$.
  Thus $r = (n'\ell - m'\ell) /2$.
  So
  \begin{equation*}
  \nu_p\left(\det G(n,m) \right) =
  \sum_{j = 1}^{\frac{n'\ell - m'\ell}{2}}
  \left(\nu_p\left(\frac{[m'\ell + j]}{[j]}\right)\right)\dim S(n, m + 2j)
  \end{equation*}
  Now, modulo $[m'\ell + j] = \pm [j]$ modulo $\ell$ and so if $\ell \nmid j$ the $p$-adic valuation of the fraction vanishes.
  Hence
  \begin{equation*}
  \nu_p\left(\det G(n,m) \right) =
\sum_{j' = 1}^{\left\lfloor\frac{n' - m'}{2}\right\rfloor}
\left(\nu_p(m' + j') - \nu_p(j')\right)\dim S(n, m + 2j'\ell)
  \end{equation*}
  Now if $p^{\nu_p(m')} > \lfloor(n' - m') / 2\rfloor$ then $\nu_p(m' + j') = \nu_p(j')$ in that range and so the total valuation vanishes.
\end{proof}

By working in an integer ring $\Z[\delta]/(m_\delta)$ where $m_\delta$ is the minimal polynomial of $\delta$ over the integers and then quotienting by a maximal ideal, we deduce that
\begin{corollary}\label{cor:when_simple}
  Let $R$ be a field of characteristic $p$ with element $\delta$ integral over the prime subfield and suppose $n,m \equiv_\ell-1$.
  Then, as $\TL^R_n(\delta)$-modules, $S(n,m)$ is irreducible if $\lfloor(n - m ) / 2\ell\rfloor < p^{\nu_p(m+1)}$.
\end{corollary}

\section{Truncation}\label{sec:truncation}
Let $y$ and $z$ be monic $(n,m)$-diagrams such that $\langle y|z \rangle = \id_m$, as (for example) in the proof of \cref{lem:form_nondegenerate}.
Recall this is possible iff $m > 0$ or $\delta \neq 0$.
Denote by $e_n^m$ the idempotent $|y\rangle\langle z|$  in $\TL_n$.
\begin{lemma}
  There is an isomorphism of algebras $\TL_m \simeq e_n^m \TL_n e_n^m $ sending $u \mapsto |y\rangle u \langle z|$.
\end{lemma}
\begin{proof}
  Firstly note that $|y\rangle u \langle z|$ is invariant under left or right multiplication by $e_n^m$ and so lies in $e_n^m \TL_n e_n^m$.
  Second, suppose $|y\rangle u \langle z| = |y\rangle u' \langle z|$.
  Then, as $\langle z|$ is epic, $|y\rangle u = |y\rangle u'$.
  But $|y\rangle$ is monic so $u=u'$ and the map $u \mapsto |y\rangle u \langle z|$ is indeed an injection.
  Finally, for any $e_n^m w e_n^m = |y\rangle \langle z| w |y \rangle \langle z |$ it is clear $\langle z| w |y \rangle \in \TL_m$.
  Hence the map is an isomorphism of $R$ spaces and $|y \rangle u_1 \langle z| \cdot |y \rangle u_2 \langle z| = |y\rangle u_1u_2\langle z|$ so it is an algebra morphism.
\end{proof}
This construction gives a truncation functor, $\mathcal{T}_n^m$, from the category of $\TL_n$ modules to the category of $\TL_m$ modules.
It sends the module $M$ to $e_n^m M\subseteq M$, with $\TL_m$ action $u \cdot (e_n^m v) = |y\rangle u \langle z| v$ and acts as restriction on morphisms.
\begin{lemma}
  This functor is exact.
\end{lemma}
\begin{proof}
  As restriction of modules, $\mathcal{T}_n^m$ is clearly left-exact.
  Suppose thus that $\namedses{f}{g}{M_1}{M_2}{M_3}$ is a short exact sequence.
  If $e_n^mm_3 \in e_n^m M_3$ then there is an $m_2 \in M_2$ with $m_3 = g(m_2)$ and $e^n_m m_3 = g(e^n_m m_2)$ and so it preserves surjective maps too.
  Exactness in the middle is clear.
\end{proof}

\begin{lemma}
  $\mathcal{T}_n^m S(n,r) \simeq S(m, r)$.
\end{lemma}
\begin{proof}
  Recall from its definition that $S(n,r) = \Hom^{\not < r}(\underline{n}, \underline{r})$.
  Then consider the map $f: \mathcal{T}_n^mS(n,r) \to S(m,r)$ sending $e_n^m\phi \mapsto \langle z | \phi$.
  This map is injective: if $\langle z|\phi = \langle z|\phi'$ then $|y\rangle\langle z|\phi = |y\rangle\langle z|\phi'$.
  It is also surjective: if $\psi \in S(m,r)$ then $|y\rangle \psi\in S(n,r)$ and this element is sent to $\langle z | y \rangle \psi = \psi$.

  It remains to show that this is a morphism of $\TL_m$ modules.
  In an exercise in Dirac notation:
  \begin{equation}
    f(u \cdot e_n^m \phi)
    = f(|y\rangle u \langle z | \phi)
    = f(e_n^m|y\rangle u \langle z | \phi)
    = \langle z|y\rangle u \langle z | \phi
    = u f(e_n^m\phi).
  \end{equation}
\end{proof}

From the previous two lemmas, one should note in particular that $\mathcal{T}_n^m$ kills all simples $D(n,r)$ for $r > m$ and refines composition series.
It sends $S(n,r)$ to $S(m,r)$, $L(n,r)$ to $L(m,r)$ and $R(n, r)$ to $R(m,r)$.
A direct consequence of this is the following:

\begin{corollary}\label{cor:simple_mult_triv}
  $[S(n,m): L(n,r)] = [S(r, m) : L(r, r)]$
\end{corollary}

Thus the problem of finding decomposition numbers reduces, by induction, to that of finding the multiplicity of the trivial module in the standard modules.
By the linkage principle of \cref{sec:linkage} we will only be concerned with standard modules in the principle linkage block.

\begin{lemma}
  Suppose that $\namedses{f}{g}{M_1}{M_2}{D(n,m)}$ is a non-split short exact sequence of $\TL_n$-modules.
  Then there is a non-split short exact sequence of $\TL_m$ modules $\namedses{f_1}{g_1}{e_n^m M_1}{e_n^m M_2}{D(m,m)}$.
\end{lemma}
\begin{proof}
  The resultant sequence is the application of $\mathcal{T}_n^m$.
  It will suffice to show that if $\namedses{f_1}{g_1}{e_n^m M_1}{e_n^m M_2}{D(m,m)}$ splits, so does $\namedses{f}{g}{M_1}{M_2}{D(n,m)}$.
  Let $h_1 : D(m,m) \to e_n^m M_2$ be a splitting $\TL_m$-morphism, so that $g_1 \circ h_1 = \id_{D(m,m)}$.

  Define $h : D(n,m) \to M_2$ as follows.
  For $x \in D(n,m)$, let $a \in \TL_n$ be such that $a \cdot y = x$.
  Then set $h(x) = a \cdot h_1(y)$.
  It is clear that this is a $\TL_n$-morphism.
  Further,
  \begin{equation}
    g\circ h(x) = g(a \cdot h_1(y)) = a \cdot g \circ h_1(y) = a \cdot g_1 \circ h_1(y) = a \cdot y = x,
  \end{equation}
  and so $\namedses{f}{g}{M_1}{M_2}{D(n,m)}$ splits.
\end{proof}
In particular, determining the nonzero $\Ext^1$ groups between simple $\TL_n$-modules is equivalent to determining which $TL_n$-modules can be extended by the trivial module.

Recall that a module of an algebra over $R$, $M$ is \emph{Schurian} if $\End(M) \cong R$.
\begin{lemma}\label{lem:standard_schurian}
  The module $S(n,m)$ is Schurian if $m > 0$ or $\delta \neq 0$.
\end{lemma}
\begin{proof}
  Construct morphisms $|y\rangle$ and $|z \rangle$ as above and let $\theta \in \End S(n,m)$.
  Then $\theta(|x\rangle) = \theta(|x\rangle\langle y|z\rangle) = |x\rangle\langle y|\theta(z)$ so it suffices to show that $\theta(|z\rangle) = \lambda_\theta |z\rangle$ for some scalar $\lambda_\theta$.
  But indeed, $|z\rangle\langle y |\theta(|z\rangle) = \theta(|z\rangle)$ and the left is a morphism factoring through the diagram $z$ so must be a linear multiple of $|z\rangle$ in $S(n,m)$.
\end{proof}

\section{Linkage}\label{sec:linkage}
In the appendix of ~\cite{ridout_saint_aubin_2014}, Ridout and Saint-Aubin develop the theory of a particular central element $F_n \in \TL^\C_n(q + q^{-1})$.
The construction given is elegant and diagrammatic, but hides some subtleties we would like to make explicit.

We define a diagrammatic notation of crossings
\begin{equation*}
  \vcenter{\hbox{\begin{tikzpicture}[scale=0.5]
\draw (0,0) -- (0,1);
\draw (1,0) -- (1,1);
\draw[very thick] (0.5,0) -- (0.5,1);
\draw[very thick] (0,0.5) -- (0.35,0.5);
\draw[very thick] (0.65,0.5) -- (1,0.5);
\end{tikzpicture}}}
 = q^{1/2}\,
\vcenter{\hbox{\begin{tikzpicture}[scale=0.5]
\draw (0,0) -- (0,1);
\draw (1,0) -- (1,1);
\draw[very thick] (0,0.5) to[out=0, in=270] (0.5,1);
\draw[very thick] (0.5,0) to[out=90, in=180] (1,0.5);
\end{tikzpicture}}}
-
 q^{-1/2}\,
\vcenter{\hbox{\begin{tikzpicture}[scale=0.5]
\draw (0,0) -- (0,1);
\draw (1,0) -- (1,1);
\draw[very thick] (0,0.5) to[out=0, in=90] (0.5,0);
\draw[very thick] (0.5,1) to[out=270, in=180] (1,0.5);
\end{tikzpicture}}}
\end{equation*}
\begin{equation*}
  \vcenter{\hbox{\begin{tikzpicture}[scale=0.5]
\draw (0,0) -- (0,1);
\draw (1,0) -- (1,1);
\draw[very thick] (0.5,0) -- (0.5,0.35);
\draw[very thick] (0.5,0.65) -- (0.5,1.0);
\draw[very thick] (0,0.5) -- (1,0.5);
\end{tikzpicture}}}
 = q^{1/2}\,
\vcenter{\hbox{\begin{tikzpicture}[scale=0.5]
\draw (0,0) -- (0,1);
\draw (1,0) -- (1,1);
\draw[very thick] (0,0.5) to[out=0, in=90] (0.5,0);
\draw[very thick] (0.5,1) to[out=270, in=180] (1,0.5);
\end{tikzpicture}}}
-
 q^{-1/2}\,
\vcenter{\hbox{\begin{tikzpicture}[scale=0.5]
\draw (0,0) -- (0,1);
\draw (1,0) -- (1,1);
\draw[very thick] (0,0.5) to[out=0, in=270] (0.5,1);
\draw[very thick] (0.5,0) to[out=90, in=180] (1,0.5);
\end{tikzpicture}}}
\end{equation*}
in order to define $F_n$ as
\begin{equation}
F_n =\,
  \vcenter{\hbox{\begin{tikzpicture}[scale=0.5]
\draw (0,0) -- (0,5);
\draw (1,0) -- (1,5);
\draw (2,0) -- (2,5);
\draw[very thick] (0.5,0) -- (0.5, 5);
\draw[very thick] (0.5,5) arc (180:0:0.5);
\draw[very thick] (0.5,0) arc (180:360:0.5);
\draw[very thick] (0,0.5) -- (0.35, 0.5);
\draw[very thick] (0,1.5) -- (0.35, 1.5);
\draw[very thick] (0,3.5) -- (0.35, 3.5);
\draw[very thick] (0,4.5) -- (0.35, 4.5);
\node at (0.25,2.7) {$\vdots$};
\node at (0.75,2.7) {$\vdots$};
\node at (1.5,2.7) {$\vdots$};
\draw[very thick] (0.65,0.5) -- (2., 0.5);
\draw[very thick] (0.65,1.5) -- (2., 1.5);
\draw[very thick] (0.65,3.5) -- (2., 3.5);
\draw[very thick] (0.65,4.5) -- (2., 4.5);

\draw[very thick] (1.5, 5) -- (1.5, 4.65);
\draw[very thick] (1.5, 4.35) -- (1.5, 3.65);
\draw[very thick] (1.5, 1.35) -- (1.5, 0.65);
\draw[very thick] (1.5, 0.35) -- (1.5, 0);
\draw[very thick] (1.5, 1.65) -- (1.5, 1.8);
\draw[very thick] (1.5, 3.35) -- (1.5, 3.2);
\end{tikzpicture}}}.
\end{equation}
This construction gives a $\Z[q^{\pm{1/2}}]$-linear span of diagrams.
To show it descends to a proper element of $\TL_n^R(\delta)$ under the usual specialisation $\delta = q + q^{-1}$ we must show that all coefficients are symmetric polynomials in $q$.
In expanding, we encounter $2n$ crossings, and so when expanded, all powers of $q$ are integral.

Since the two crossings are $\iota$-images of each other, $\iota(F_n) = F_n$.
However, the map $q^{1/2}\mapsto q^{-1/2}$ also swaps the crossing diagrams so $F_n$ is fixed by this map.
Hence all coefficients in the expression of $F_n$ in the diagram basis are \emph{symmetric polynomials} in $q$ and $F_n$ is a well defined element of $\TL^{\Z[\delta]}_n(\delta)$.
Thus $F_n$ descends to a (possibly zero) element of $\TL^R_n(\delta)$ for any $\delta$ and $R$.

The proof of the following is performed in $\TL_n^{\Z[\delta]}(q+q^{-1})$ and so descends to any ring by the evaluation principle.
\begin{proposition}\cite[A.1]{ridout_saint_aubin_2014}
  The element $F_n$ is central.
\end{proposition}
Similarly the following holds in any ring
\begin{proposition}\cite[A.2]{ridout_saint_aubin_2014}
  The element $F_n$ acts on $S(n,m)$ by the scalar $\Delta_{m+1}$.
\end{proposition}

This gives us a measure of linkage on standard modules.
\begin{theorem}\label{thm:linkage}
  Let $\delta \neq \pm 2$.
  If the module $D(n,m)$ appears as a composition factor of $S(n,m')$ then $m > m'$ and $m$ and $m'$ lie in the same orbit of $D_\infty$ on $\Z$ as described in \cref{sec:quantum:difference}.
\end{theorem}

Note that $\TL_{n-1}$ injects into $\TL_n$ by the addition of a single through strand at the bottom.
We will denote the restriction of a $\TL_n$-module $M$ to a $\TL_{n-1}$ along this injection by $M\restr$.

\begin{corollary}\cite[4.2]{ridout_saint_aubin_2014}
  If $m \not\equiv_\ell -1$, then $S(n,m)\restr \cong  S(n-1, m-1) \oplus S(n-1, m+1)$.
\end{corollary}

This gives us partial results on the dimensions of simple modules.
Indeed, if the restriction splits, then the Gram matrix can be put in block diagonal form.
However, when $m\equiv_\ell -2$, one of the blocks vanishes completely~\cite[\S5]{westbury_1995}.
\begin{proposition}\label{prop:incremental_result_1}
  If $m\not\equiv_\ell -1$ then 
  \begin{equation}
    \dim D(n,m) = \begin{cases}
      \dim D(n-1, m-1) & m\equiv_\ell-2\\
      \dim D(n-1, m-1) + \dim D(n-1,m+1) & \text{else}
    \end{cases}
  \end{equation}
  If, further, $R$ is a characteristic zero field, and $m \equiv_\ell -1$, $\dim D(n,m) = \dim S(n,m)$.
\end{proposition}
The sequel of this paper is concerned predominantly with filling in the final case of a field $R$ of characteristic $p>0$ and $m\equiv_\ell -1$.


\section{Morphisms between Standard Modules}\label{sec:morphisms}
The morphisms between standard modules are particularly nicely behaved.

\subsection{Uniqueness}\label{sec:morphism_unique}
In this section we prove that the space of morphisms $\Hom_{\TL_n}(S(n,r), S(n,s))$ is at most one dimensional.

\begin{lemma}\label{lem:morphism_space}
  If $r > 0$ or $\delta \neq 0$, all morphisms $\theta \in \Hom_{\TL_n}(S(n,r), S(n,s))$ are of the form
  \begin{equation}
    u \mapsto u v
  \end{equation}
  for some $v : \underline{r} \to \underline{s}$ a linear combination of monic diagrams.
  Further, every such $\theta$ has cyclic image.
\end{lemma}

\begin{proof}
  If $r = 0$, then the only value of $s$ for which there are homomorphisms is also $0$ and we know that $S(n,r)$ is Schurian so picking $v=\id_r$ suffices.
  If $r \neq 0$, then there are $y,z : \underline{n} \to \underline{r}$ such that $\langle y | z \rangle = \id_r$.

  Then, for any monic $u : n \to r$,
  \begin{equation}\label{eq:monic_argument}
    \theta(u) = \theta(|u\rangle\langle y|z\rangle) = |u\rangle\langle y|\theta(z) = |u\rangle\Big(\langle y|\theta(z)\Big).
  \end{equation}
  Thus it is clear from the second-to-last equality that $\theta(z)$ generates $\im\theta$ as a $\TL_n$-module.

  Let $v' = \langle y | \theta(z)$.
  Suppose a $\underline{n} \to \underline{s}$ diagram in the support of $\theta(x)$ is not monic.
  Then it is the composition of a diagram in the support of $x$ and one in the support of $v'$.
  Since the diagrams in the support of $u$ are all monic, we must have that the latter is not monic.
  But then it can be removed from $v'$ and the composition will be unchanged modulo diagrams factoring through numbers less than $s$.
  Hence $v$ can be chosen by removing all diagrams from the support of $v'$ which are not monic and in $S(n,s)$ we will still have $\theta(x) = x v$.
\end{proof}

\begin{remark}\label{rem:killed_morphisms}
  Recall the morphisms $y$ and $z$ from \cref{lem:form_nondegenerate}.
Any $u_i$ for $1\le i < r$ acts as zero on $z$ in $S(n,r)$.
Hence $u_i\theta(z) = \theta(u_i z) = 0$.
Since $\theta(z) = z v$ for some $v$, we see that indeed $u_i v = 0$ for all $1 \le i < r$.
This is to say that $v$ must span a trivial submodule of $S(r, s)$.
\end{remark}

To show that each $S(n, m)$ has at most one trivial submodule, we must introduce a partial order on diagrams.
The bijection between two-part standard tableaux and standard diagrams in \cref{lem:diagram-tableaux} allows us to label diagrams by tableau.
Let $\Omega(n, m)$ be the set of standard tableaux of shape $\left((n + m)/2, (n - m)/2\right)$ and so be in bijection with monic $\underline{n} \to \underline{m}$ diagrams.
For any two part tableau $t$, let $t^{(1)}$ be the tuple of the first row of $t$ and $t^{(2)}$ that of the second.
Define $\le$ on $\Omega(n,m)$ by considering the partial order generated by $t < s$ if $t$ and $s$ differ in exactly two places and $t^{(2)}$ is lexicographically less than $s^{(2)}$.

We demonstrate the partial order on $\Omega(6,2)$ with dark arrows below:

\begin{center}
  \begin{tikzpicture}
    \node (24) at (0,0) {\young(1356,24)};
    \node (25) at (-2,2) {\young(1346,25)};
    \node (26) at (-2,4) {\young(1345,26)};
    \node (34) at (2,2) {\young(1256,34)};
    \node (35) at (2,4) {\young(1246,35)};
    \node (36) at (-2,6) {\young(1245,36)};
    \node (45) at (2,6) {\young(1236,45)};
    \node (46) at (0,8) {\young(1235,46)};
    \node (56) at (0,10) {\young(1234,56)};

    \draw[very thick, ->] (46) -- (56);
    \draw[very thick, ->] (36) -- (46);
    \draw[very thick, ->] (45) -- (46);
    \draw[very thick, ->] (26) -- (36);
    \draw[very thick, ->] (35) -- (36);
    \draw[very thick, ->] (35) -- (45);
    \draw[very thick, ->] (25) -- (26);
    \draw[very thick, ->] (25) -- (35);
    \draw[very thick, ->] (34) -- (35);
    \draw[very thick, ->] (24) -- (25);
    \draw[very thick, ->] (24) -- (34);
  
    \draw[thin, blue, ->] (56) edge[out=-70, in=70] (46);
    \draw[thin, blue, ->] (45) edge[out=110, in=-30] (46);
    \draw[thin, blue, ->] (35) edge[out=130, in=-10] (36);
    \draw[thin, blue, ->] (25) edge[out=70, in=-70] (26);
    \draw[thin, blue, ->] (26) edge[out=170, in=-170] (26);
    \draw[thin, blue, ->] (36) edge[out=170, in=-170] (36);
    \draw[thin, blue, ->] (46) edge[out=170, in=-170] (46);

    \draw[thin, green, ->] (56) edge[out=170, in=-170] (56);
    \draw[thin, green, ->] (25) edge[out=170, in=-170] (25);
    \draw[thin, green, ->] (35) edge[out=170, in=-170] (35);
    \draw[thin, green, ->] (46) edge[out=110, in=-110] (56);
    \draw[thin, green, ->] (34) edge[out=110, in=-110] (35);
    \draw[thin, green, ->] (24) edge[out=110, in=-30] (25);

    \draw[thin, orange, ->] (46) edge[out=160, in=-160] (46);
    \draw[thin, orange, ->] (36) edge[out=70, in=-140] (46);
    \draw[thin, orange, ->] (24) edge[out=170, in=-170] (24);
    \draw[thin, orange, ->] (34) edge[out=-160, in=70] (24);
    \draw[thin, orange, ->] (35) edge[out=70, in=-70] (45);
    \draw[thin, orange, ->] (45) edge[out=10, in=-10] (45);

    \draw[thin, red, ->] (24) edge[out=30, in=-100] (34);
    \draw[thin, red, ->] (34) edge[out=10, in=-10] (34);
    \draw[thin, red, ->] (25) edge[out=20, in=-140] (35);
    \draw[thin, red, ->] (35) edge[out=10, in=-10] (35);
    \draw[thin, red, ->] (26) edge[out=70, in=-70] (36);
    \draw[thin, red, ->] (36) edge[out=160, in=-160] (36);

    \draw[thin, purple, ->] (24) edge[out=160, in=-160] (24);
    \draw[thin, purple, ->] (25) edge[out=160, in=-160] (25);
    \draw[thin, purple, ->] (26) edge[out=160, in=-160] (26);
    \draw[thin, purple, ->] (36) edge[out=-110, in=110] (26);
    \draw[thin, purple, ->] (34) edge[out=-110, in=40] (24);
    \draw[thin, purple, ->] (35) edge[out=-130, in=10] (25);
    \draw[thin, purple, ->] (45) edge[out=-130, in=90] (24);

  \end{tikzpicture}
\end{center}

\begin{lemma}
  The above order is given by $t \le s$ iff $t^{(2)}_i \le s^{(2)}_i$, for each $i$.
\end{lemma}

For any $t \in \Omega(n,m)$, let $\underline{t}$ be the diagram in $S(n, m)$.
The following is a critical lemma for many induction proofs.

\begin{lemma}\label{lem:dot_induction}
  For any $1 \le i < n$ and $t \in \Omega(n,m)$, $u_i \cdot \underline{t}$ is either $\delta \underline t$ or a diagram morphism $\underline{s}$.
  In the latter case $t\neq s$ and there is no $r \in \Omega(n,m)$ such that $t < r < s$.
  Further if $u_i \cdot \underline{t} = \underline{s} > t$ then there is no other $r < s$ such that $u_i \cdot \underline{r} = \underline{s}$.
\end{lemma}
In the above diagram, we have indicated the action of the $u_i$ by different light coloured arrows.
The crux of this lemma is that such arrows never ``go up by more than one'' and if they do go upwards, they are the only arrow of that colour incident to that tableau.
\begin{proof}
  We will identify elements of $\Omega(n,m)$ with $(n+m)/2$-length tuples representing their lower row.
  Thus the least element of $\Omega(8,2)$ is represented $(2,4,6)$ and the greatest $(6,7,8)$.

  \textbf{Case 1:}
  If $i \in t^{(1)}$ and $i+1 \in t^{(2)}$, then the diagram $\underline{t}$ has a simple cap at $i$ so $u_i \cdot \underline t = \delta \underline{t}$.

  \textbf{Case 2:}
  If $i \in t^{(2)}$ and $i+1 \in t^{(2)}$, then suppose that in $t$, the site at $i$ is connected to the site at $c_0$.
  If $c_-$ and $c_+$ are the elements of $t^{(2)}$ immediately preceding and succeeding $c_0$, then the action of $u_i$ sends
  \begin{equation}
    t = \underline{(k_1,\ldots,c_-,c_+ \ldots, i, i+1, \ldots, k_r)} \mapsto \underline{(k_1, \ldots, c_-, c_0, c_+ \ldots, i+1, \ldots, k_r)} = s
  \end{equation}
  This is clearly lexicographically smaller as $c_0 < c_+$.
  Hence if $s$ and $t$ are comparable $s < t$.

  \textbf{Case 3:}
  If $i \in t^{(1)}$ and $i+1 \in t^{(1)}$ then let $c_0$ be the first site after $i+1$ which lies in the second row of $t$.
  Then
  \begin{equation}
    t= \underline{(k_1,\ldots,c_0, \ldots, k_r)} \mapsto \underline{(k_1, \ldots, i+1, c_0, \ldots, k_r)} = s
  \end{equation}
  where some other site has been removed from the tuple after $c_0$ (it is the site to which $i+1$ is connected).
  Again, $s$ is lexicographically less than $t$.

  \textbf{Case 4:}
  In the final case, $i \in t^{(2)}$ and $i+1 \in t^{(1)}$.
  Here it is clear that
  \begin{equation}
    \underline{(k_1,\ldots,i, \ldots, k_r)} \mapsto \underline{(k_1, \ldots, i+1, \ldots, k_r)}
  \end{equation}
  which is an increase in the partial ordering.
  However, note that each covering strictly increases the value in the $j$-th position of the tuple for some $j$.
  Hence there is no other tuple lying between $t$ and $u_i \cdot t$.

  In case 2 and 3, $u_i$ strictly decreased the tableau under the partial order.
  Hence if there were another $u < s$ such that $u_i \cdot \underline{u} = \underline{s}$ it would have to be under case 4.
  However, it is clear that this cannot occur unless $u = t$ as desired.
\end{proof}

Let $t^-$ be the unique minimum of $\Omega(n,m)$, and $t^+$ the unique maximum.
For example, in $\Omega(9,1)$,
\begin{equation}
  t^- = \young(13579,2468)
  \quad\quad\quad
  t^+ = \young(12345,6789)
\end{equation}
We can now prove an important result about the vanishing of coefficients.
\begin{lemma}
  If $u_i \cdot z = 0$ for all $1 \le i < n$ and some $0 \neq z \in S(n, m)$, then $z$ is uniquely (linearly) determined by its coefficient at $t^+$.
\end{lemma}
\begin{proof}
  We will induct the statement ``the coefficient of the diagram $t$ is uniquely (linearly) determined by that of $t^+$'' using the partial order on $\Omega(n,m)$.
  More precisely, we will show that for each upwards closed subset of $\Omega(n,m)$, $S$, the coefficients of the diagrams in $S$ are determined by that of $t^+$ and induct using inclusion.
  Our base case will be $S = \{ t^+ \}$ which is clear.

  Now, suppose $S \neq \Omega(n,m)$.
  Then there is an element $s \not\in S$ such that there is an $i$ with $u_i \cdot s \in S$.
  Indeed, there is an element $s \not\in S$ such that no $t > s$ exists which is also not in $S$.
  Then $s$ and $t$ differ in exactly two places and $s^{(2)} < t^{(2)}$.
  The two places $s$ and $t$ differ must be labelled $i$ and $i+1$ for some $i$. and so we find ourselves in case 4 of the above proof, where $u_i \cdot \underline{t} = \underline{s}$.

  Let $z = \sum_{u \in \Omega(n,m)} z_u u$.
  Consider now the coefficient of $t$ in $u_i \cdot z$.
  By \cref{lem:dot_induction} it is exactly
  \begin{equation}
    z_s + \delta z_t + \sum_{u \in U} z_u
  \end{equation}
  where $U \subseteq S$.
  Since this vanishes (as $u_i \cdot z = 0$), we see $z_s$ is completely determined by the values in $S$ and so we are done.
\end{proof}

In particular, there is only one (up to scaling) element of $S(n,m)$ killed by $\mathcal{F}^{n-2}(\TL_n)$.
\begin{corollary}\label{cor:single_trivial_submodule}
  There is at most one submodule of $S(n, m)$ isomorphic to the trivial module $S(n,n)$.
\end{corollary}

This completes the ingredients for the proof of the claim made at the beginning of the section.

\begin{corollary}
  For $r> s$, the $R$-space $\Hom_{\TL_n}(S(n,r), S(n,s))$ is at most one-dimensional.
\end{corollary}
\begin{proof}
  Every element of $\Hom_{\TL_n}(S(n,r), S(n,s))$ is identified with a morphism $v : \underline{r} \to \underline{s}$ by \cref{lem:morphism_space}.
  If post-composition by this $v$ is a $\TL_n$-morphism from $S(n,r)$ to $S(n,s)$ then any such $v$ spans a trivial submodule of $S(r, s)$ by \cref{rem:killed_morphisms}.  But by \cref{cor:single_trivial_submodule}, we have the result.
\end{proof}

\subsection{Candidates and Composition}
We will denote by $v_{r,s}$ the element of $S(r,s)$ given by the formula
\begin{equation}
  v_{r,s} = \sum_{x} h_{F(x)} x
\end{equation}
where the sum runs over all monic $(r,s)$-diagrams $x$, and $h_{F(x)}$ is the polynomial described in \cref{prop:h_F,def:h_F_x}.
Note that since $h_{F(x)}$ is a polynomial in $\delta$, we can consider this morphism over any ring.
Further, since the diagram $\id_{\frac{r+s}{2}} \otimes \cap^{\frac{r-s}{2}}$ has a totally ordered forest, its coefficient is 1.
Hence $v_{r,s} \neq 0$.
These will be called ``candidate morphisms''.
Through post-composition these engender linear morphisms between all $S(n,r)$ and $S(n,s)$.

\begin{proposition}\label{prop:jwform}\cite[3.6]{graham_lehrer_1998}
  If $s < r < s + 2p^{(k)}$ and $s + r \equiv_{2p^{(k)}} -2$ then the map $x \mapsto x \circ v_{r,s}$ is a morphism of $\TL_n$ modules for every $n$.
\end{proposition}
Readers who follow the proof of Graham and Lehrer for \cref{prop:jwform} should be cautioned that the first line should ask for the map $\mu : t \to s$ to be \emph{standard} instead of monic.
By the argument in \cref{sec:morphism_unique}, we can phrase \cref{prop:jwform} as constructing the trivial submodules of $S(r,s)$ for certain $r$ and $s$.

Despite not all of these candidate morphisms being $\TL_n$ morphisms, some compositions of them are.
We investigate composition of these morphisms below.

\begin{proposition}
  Let $a \ge b \ge c$ be naturals of the same parity.
  Then
  \begin{equation}
    v_{a,b}\circ v_{b,c} =\gaussianquant{\frac{a-c}{2}}{\frac{a-b}{2}} v_{a,c}.
  \end{equation}
\end{proposition}
\begin{proof}

The maximum monic diagram $\underline{a} \to \underline{c}$ factors through $\underline{b}$ only by the following diagrams.

\begin{equation}
  \vcenter{\hbox{
  \begin{tikzpicture}
    \draw (0,5) -- (0,-1.5);
    \draw[very thick] (0,4.5) -- (1,4.5);
    \draw[very thick] (0,3.5) -- (1,3.5);
    \draw[very thick] (0,3.3) -- (1,3.3);
    \draw[very thick] (0,2.3) -- (1,2.3);

    \draw[very thick] (0,2.1) arc (90:-90:1);
    \draw[very thick] (0,1.9) arc (90:-90:0.8);
    \draw[very thick] (0,1.3) arc (90:-90:0.2);

    \draw[very thick] (0,-.1) -- (1,-.1);
    \draw[very thick] (0,-1.1) -- (1,-1.1);

    \node at (0.5, 4.1) {$\vdots$};
    \node at (0.5, 2.9) {$\vdots$};
    \node at (0.5, -.5) {$\vdots$};
    \node at (0.49, 1.1) {$\cdots$};

    \draw [decorate,decoration={brace,amplitude=5}] (-.1,3.45) -- (-.1,4.55);
    \draw [decorate,decoration={brace,amplitude=5}] (-.1,2.25) -- (-.1,3.35);
    \draw [decorate,decoration={brace,amplitude=5}] (-.1,0.05) -- (-.1,2.15);
    \draw [decorate,decoration={brace,amplitude=5}] (-.1,-1.15) -- (-.1,-.05);

    \node at (-1, 4.1) {$c$};
    \node at (-1, 2.9) {$\frac{b-c}{2}$};
    \node at (-1, 1.1) {$a-b$};
    \node at (-1, -.5) {$\frac{b-c}{2}$};
  \end{tikzpicture}}}
    \quad\quad\quad
    \quad\quad\quad
  \vcenter{\hbox{
  \begin{tikzpicture}
    \draw (0,3.5) -- (0,-0.1);
    \draw[very thick] (0,3.3) -- (1,3.3);
    \draw[very thick] (0,2.3) -- (1,2.3);

    \draw[very thick] (0,2.1) arc (90:-90:1);
    \draw[very thick] (0,1.9) arc (90:-90:0.8);
    \draw[very thick] (0,1.3) arc (90:-90:0.2);

    \node at (0.5, 2.9) {$\vdots$};
    \node at (0.49, 1.1) {$\cdots$};

    \draw [decorate,decoration={brace,amplitude=5}] (-.1,2.25) -- (-.1,3.35);
    \draw [decorate,decoration={brace,amplitude=5}] (-.1,0.05) -- (-.1,2.15);

    \node at (-1, 2.9) {$c$};
    \node at (-1, 1.1) {$b-c$};
  \end{tikzpicture}}}
\end{equation}
They have coefficients $\gaussianquant{\frac{a-c}{2}}{\frac{b-c}{2}}$ and 1 respectively.
Thus if indeed $v_{a,b}\circ v_{b,c}$ is a multiple of $v_{a,c}$, it must have the given coefficient.
The proof that the composition of candidate morphisms is a candidate morphism (up to a scalar) can be found in~\cite{cox_graham_martin_2003} where an infinite tower of algebras is employed.
\end{proof}

\section{Decomposition Numbers}\label{sec:decomposition}
We first prove a lower bound on the decomposition numbers of $S(n,m)$.

\begin{theorem}\label{thm:lower}
  Let $n + 1 = [n_i, n_{i-1}, \ldots, n_0]_{p,\ell}$.
  Then the standard module $S(n,m)$ has a trivial composition factor if $m$ lies in
  \begin{equation}
\supp(n) = \{n_i p^{(i)} \pm n_{i-1}p^{(i-1)} \pm \cdots \pm n_1p^{(1)} \pm n_0p^{(0)} -1 \}.
  \end{equation}
\end{theorem}
Before proceeding to the proof of \cref{thm:lower} it is worth considering the structure of $\supp(n)$.
In \cref{fig:in} we show the sets $\supp(n)$ for $0 \le n \le 700$ where $\ell = 5$ and $p = 3$.
A pixel in the $m$-th column of the $n$-th row is black iff $m \in \supp(n)$.
We note it exhibits a beautiful fractal-like structure.

\begin{figure}[htpb]
\begin{center}
  \includegraphics[width=0.5\textwidth]{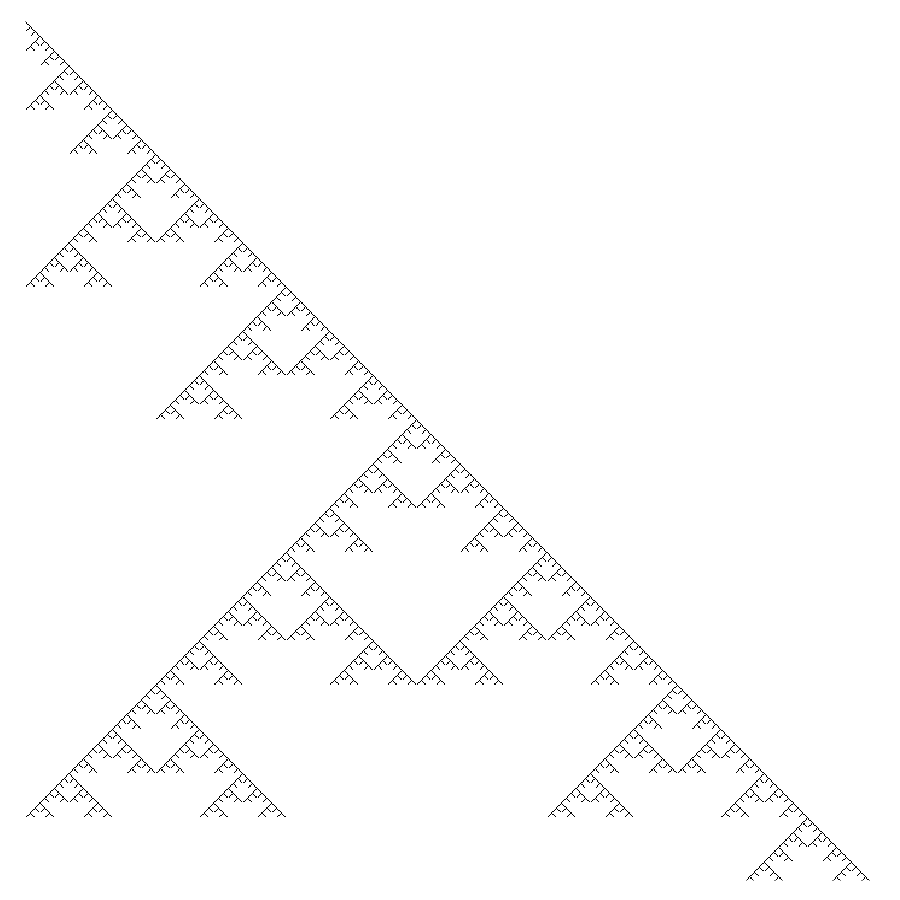}
\end{center}
\caption{Values of $\supp(n)$ for $\ell = 5$ and $p=3$.}%
\label{fig:in}
\end{figure}

As in the theory for the algebraic group $SL_2$, we define a $k$-wall to be a number of the form $ap^{(k)}-1$.

\begin{lemma}
  Suppose that $a p^{(k)} -1 \le m\le n < (a+1)p^{(k)}-1$.
  Then if $m' = 2ap^{(k)}-2 - m$ is the reflection of $m$ about the $k$-wall beneath it, $m \in \supp(n) \Rightarrow m' \in \supp(n)$.
\end{lemma}
\begin{proof}
  Write $n+1 = \sum_{j=0}^k n_{j}p^{(j)}$ and let $m+1 = \sum_{R^+}n_{j}p^{(j)} - \sum_{R^-}n_{j}p^{(j)}$ for $R^+ \sqcup R^- = \{0,\ldots,k\}$ and $k \in R^+$ with $n_k = a$.
  Then $k-1 \in R^+$ lest $m < a p^{(k)}-1$ and 
  \begin{align}
    m' &= 2ap^{(k)} - \sum_{j \in R^+}n_j p^{(j)} + \sum_{j \in R^-} n_j p^{(j)}-1\nonumber\\
    &= \sum_{j \in R^-\cup\{k\}} n_j p^{(j)}- \sum_{j \in R^+\setminus\{k\}}n_j p^{(j)} - 1 \in \supp(n)
  \end{align}
\end{proof}

We now return to \cref{thm:lower}.

\begin{proof}
  The reflection of $n+1$ about the line $n_{i}p^{(i)}$ is $-\sum_{r=0}^{i-1}n_{j}p^{(j)} + n_{i}p^{(i)}$.
  Hence if $m_{i}$ is the reflection of $n$ about the $i$-wall below $n$, this is the value of $m_{i} + 1$.
  Note that $m_{i}$ is the least value in $\supp(n)$.

  Similarly, if $m_j$ is the reflection of $n$ about the $j$-wall below $n$, then
  \begin{equation}\label{eq:mjdef}
    m_j+1 = -\sum_{r = 0}^{j-1}n_{r}p^{(r)} + \sum_{r=j}^i n_r p^{(r)}
  \end{equation}

  Now, $n - m_j = 2\sum_{r = 0}^{j-1}n_{r}p^{(r)} < 2 p^{(j)} = 2\ell p^{j-1}$ and $n + m_{j} = 2\sum_{r=j}^n n_{r}p^{(r)} - 2$.
  Thus by \cref{prop:jwform} $v_{n,m_j}$ encodes a nonzero $\TL_n$-morphism.
  Hence each $S(n, m_j)$ has a trivial composition factor in the image of $v_{n, m_j}$, namely a trivial submodule.

  Now suppose $m \in \supp(n)$.
  Let $R_m^+ \sqcup R_m^-$ be the partition of $\{0, \ldots, i\}$ such that $m + 1 = \sum_{j\in R_m^+}n_{j}p^{(j)} - \sum_{j\in R_m^-}n_{k}p^{(j)}$ and let $\epsilon_j^m = +$ if $j \in R_m^+$ and $-$ if $j \in R_m^-$.
  We can now define a partial order on such $m$ where $m \prec m'$ iff $R_m^+ \subset R_{m'}^+$.
  In this order, $n$ is the unique maximal element and $m_{i}$ is the unique minimal element.
  This order is refined by the usual order on the naturals.

  Suppose $m \prec m'$.
  Then $m' - m = 2\sum_{j \in R_{m'}^- \setminus R_m^-} n_j p^{(j)}$.
  If further $m' \prec m''$, it is clear that in the $(\ell, p)$-adic decomposition of $(m''-m')/2$ and $(m''-m)/2$, the digits of the former are at least as large as the digits of the latter and hence by \cref{thm:lucas} we see that
  \begin{equation}
    \gaussianquant{\frac{m''-m}{2}}{\frac{m''-m'}{2}}
  \end{equation}
  is nonzero.
  Thus the composition of $v_{m'',m'}$ and $v_{m', m}$ is a nonzero multiple of $v_{m'', m}$.

  In particular, let $m'_{j}$ be the reflection of $m_{i}$ around the $j$-wall \emph{above} $m_{i}$.
  That is,
  \begin{equation}\label{eq:m'jdef}
    m'_j + 1 = \sum_{r = 0}^{j-1} n_{r}p^{(r)} - \sum_{r = j}^{i-1}n_{r}p^{(r)} + n_{i}p^{(i)}
  \end{equation}
  Similarly to the above argument, we see $v_{m'_j,m_i}$ encodes a nonzero $\TL_n$ morphism.
  Its kernel is thus a submodule of $S(n, m'_j)$.
  However, the image of $v_{n, m'_j}$ is a one dimensional linear subspace of $S(n, m'_j)$ and the quotient by the kernel of $v_{m'_j, m_i}$ is isomorphic to the image of $v_{n,m_i}$ --- the trivial $\TL_n$-module.
  Hence the trivial module appears as a composition factor of $S(n,m'_j)$.

  It remains to show the result for all members of $\supp(n)$ not of the form given in \cref{eq:m'jdef,eq:mjdef}.
  The proof is by induction, and \cref{eq:m'jdef,eq:mjdef} can be considered the base cases (although the base case of $n$ and $m_1$ would also be sufficient).

  We will make the slightly stronger statement that not only do these $S(n,m)$ have a trivial composition factor, but that this factor is ``embodied'' by $v_{n, m}$.
  That is to say that there is a $\TL_n$-submodule, $N(n,m)$, of $S(n,m)$ such that the image of $v_{n,m}$ is isomorphic to the trivial module in $S(n,m)/N(n,m)$.
  In the above,  $N(n,m_j)$ is the zero submodule and $N(n, m'_j) = \ker v_{m'_j, m_i}$.

  Recall the definitions of the ``signs'', $\epsilon^m$ such that
  \begin{equation}
    m+1 = \sum_{j=0}^i \epsilon^m_j n_j p^{(j)}.
  \end{equation}
  For example, the tuple $\epsilon^n = (+,+,\ldots,+)$, and $\epsilon^{m_i} = (-,-,\ldots,-,+)$.
  Note that $\epsilon^m_i = +$ for all $m \in \supp(n)$.

  We will now state our inductive step before tying it all in together in a well-formed induction argument.
  Suppose that the stronger result holds for $m \prec m'$ where $\epsilon^{m}_j = -$ for all $0\le j \le t$ and $\epsilon^{m'}_j = +$ for all $0\le j \le t$.
  Further, suppose $\epsilon^{m} = \epsilon^{m'}$ for larger indices.

  Then $m$ and $m'$ are reflections of each other about a $t$-wall and so $v_{m',m}$ is a $\TL_n$-morphism (to be precise: $m' - m = 2\sum_{j=0}^t n_{j}p^{(j)} < 2p^{(t+1)}$ and their sum is $2\sum_{j=t+1}^i n_j p^{(j)} - 2$ so the conditions of \cref{prop:jwform} are met).

  Since $m \prec m' \prec n$, we see that $v_{m', m} \circ v_{n,m'}$ is a nonzero multiple of $v_{n,m}$.
  As before, for any $m \prec m'' \prec m'$, a reflection of $m'$ about a $s$-wall for $s < t$, we see that $v_{m', m''}$ must incur a trivial factor in its image.
  Similarly if $m''$ is a reflection of $m$ over a $s$-wall, the kernel of $v_{m'', m}$ forms $N(n, m'')$.

  The trivial factor in $m''$ is embodied (in either case) by $v_{m', m''} \circ v_{n,m'}$.
  Again, as $m''\prec m' \prec n$, this is a nonzero multiple of $v_{n, m''}$, so the stronger statement holds for all such $m''$.

  We are at last ready to state our induction.
  For any $m \in \supp(n)$, let the ``twistiness'' of $m$ be the number of times $\epsilon^m$ changes sign when read.
  That is, the twistiness of $m$ is the number of indices $j$ such that $\epsilon^m_j \neq \epsilon^m_{j+1}$.

  The numbers $n$ and $m_i$ each have twistiness 1 and form our base case.
  Suppose that $m\in \supp(n)$ and the result is known for all numbers of lesser twistiness.
  Let $t$ be the smallest index such that $\epsilon^m_t \neq \epsilon^m_{t+1}$ and $\tilde m$ be that element such that $\epsilon^{\tilde m}_j = \epsilon^m_{t+1}$ for all $j \le t$ and matches $\epsilon^m$ elsewhere.
  Then the twistiness of $\tilde m$ is one less than that of $m$.
  If we repeat the process to $\tilde m$ to get $\tilde m'$ then either $\tilde m' \prec m \prec \tilde m$ or $\tilde m \prec m \prec \tilde m'$.

  \begin{figure}[htpb]
  \begin{center}
    \begin{tikzpicture}[scale=0.4]
  \foreach \i in {0,...,31}
  {
    \fill (\i, 0) circle (0.1);
  };
  \foreach \i in {0,...,15}
  {
    \draw (2*\i+0.5, 1) -- (2*\i, 0)  node[midway, left=-2pt]{\tiny$-$};
    \draw (2*\i+0.5, 1) -- (2*\i+1, 0)  node[midway, right=-2pt]{\tiny$+$};
  };
  \foreach \i in {0,...,7}
  {
    \draw (4*\i+1.5, 2) -- (4*\i+0.5, 1) node[midway, left=-1pt]{\tiny$-$};;
    \draw (4*\i+1.5, 2) -- (4*\i+2.5, 1) node[midway, right=-1pt]{\tiny$+$};;
  };
  \foreach \i in {0,...,3}
  {
    \draw (8*\i+3.5, 3) -- (8*\i+1.5, 2) node[midway, left]{\tiny$-$};
    \draw (8*\i+3.5, 3) -- (8*\i+5.5, 2) node[midway, right]{\tiny$+$};
  };
  \foreach \i in {0,...,1}
  {
    \draw (16*\i+7.5, 4) -- (16*\i+3.5, 3) node[midway, above=-1.5pt]{\tiny$-$};
    \draw (16*\i+7.5, 4) -- (16*\i+11.5, 3) node[midway, above=-1.5pt]{\tiny$+$};
  };
  \foreach \i in {0}
  {
    \draw (32*\i+15.5, 5) -- (32*\i+7.5, 4) node[midway, above]{$-$};
    \draw (32*\i+15.5, 5) -- (32*\i+23.5, 4) node[midway, above]{$+$};
  };
  \node at (10, -.5) {\tiny$m''$};
  \node at (11, -.5) {\tiny$m'$};
  \node at (8, -.5) {\tiny$m$};
  \node at (31, -.5) {\tiny$n$};
  \node at (0, -.5) {\tiny$m_5$};
  \node at (16, -.5) {\tiny$m_4$};
  \node at (24, -.5) {\tiny$m_3$};
  \node at (28, -.5) {\tiny$m_2$};
  \node at (30, -.5) {\tiny$m_1$};
  \node at (15, -.5) {\tiny$m'_4$};
  \node at (7, -.5) {\tiny$m'_3$};
  \node at (3, -.5) {\tiny$m'_2$};
  \node at (1, -.5) {\tiny$m'_1$};
\end{tikzpicture}
  \end{center}
  \caption{An example of the elements of $\supp(n)$, highlighting the elements $m_i$, $m'_j$ and a triple $(m, m', m'')$ from the proof of \cref{thm:lower}.}%
  \label{fig:binary_tree}
  \end{figure}
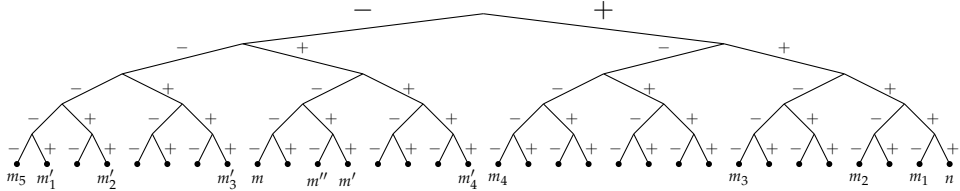

  The conditions of the inductive step are clearly met, and so the result holds for $m$.
\end{proof}

We now show the corresponding upper bound.

\begin{theorem}\label{thm:upper}
  The multiplicity of the trivial module as a composition factor of $S(n,m)$ is 1 if $m \in \supp(n)$ and is zero otherwise.
\end{theorem}
\begin{proof}
  The lower bound has been proven in \cref{thm:lower} and so all that remains is the upper bound.
  We prove the bound (and hence the result) by induction on $n$ and $n-m$.
  Thus assume the result is known for all $S(n',m')$ for $n' < n$ and for $m' > m$ when $n'=n$.

  The first key observation is that if the trivial module appears as a composition factor of $S(n,m)$, it must appear with at least that multiplicity in the restriction $S(n,m)\restr$.
  This module has a filtration with factors $S(n-1, m-1)$ and $S(n-1,m+1)$ for which we know the multiplicity of the trivial module by induction.

  Now let us assume that $S(n,m)$ contains at least one trivial composition factor.
  By \cref{thm:linkage} we must have that $m + 1 \equiv_{2\ell} \pm(n+1)$.

  Suppose $m + 1 \equiv_{2\ell} n+1$ and $n+1 \not \equiv_\ell 0$.
  Then $m \equiv_{2\ell} n$ but $m + 2 \not \equiv_{2\ell} n$ and $m+2 \not \equiv_{2\ell} -n$.
  Thus only $S(n-1, m-1)$ has a trivial composition factor and it appears only once.
  Write $n = \sum_{j=0}^k n_j p^{(j)}$ and $m = \sum_{j\in R^+} n_j p^{(j)} - \sum_{j\in R^-} n_j p^{(j)}$.
  This is possible by the inductive assumption.
  Then $0 \in R^+$ as $m - n\equiv_{2\ell} 0$ and $0 \le n_0 < \ell-1$.
  Thus $m = (n_0 + 1) + \sum_{j \in R^+\setminus\{0\}} n_j p^{(j)} - \sum_{j\in R^-} n_j p^{(j)} -1 \in \supp(n)$.

  On the other hand, if $m + 1 \equiv_{2\ell} -(n+1)$ and $n+1 \not \equiv_\ell 0$, then $m + 2 \equiv_{2\ell} n $ but $m \not \equiv_{2\ell} n$ and $m \not \equiv_{2\ell}-n$ and so should the trivial must be a factor of $S(n-1, m+1)$.
  The argument runs similarly to the above to show that the multiplicity is at most one and that $m \in \supp(n)$.

  Hence we have proved the result for $\ell \nmid n+1$.
  Henceforth we may assume that $\ell \mid n+1$ and so $\ell \mid m+1$.

  Write $n = \sum_{j = o}^k {n}_j p^{(j)}$ and $n+1 = \sum_{j = 1}^k \tilde{n}_j p^{(j)}$.
  Then by restriction to $n-1$ and the range of $\supp(n-1)$ we see $m+1 \ge \tilde{n}_k p^{(k)} - \sum_{j = 1}^{k-1} \tilde{n}_j p^{(j)} = \min\supp(n)$.
  Let
  \begin{equation}
    \supp(n-1) = I_{n-1}^i \supset I_{n-1}^{i-1} \supset \cdots \supset I_{n-1}^0 = \{n-1\}
  \end{equation}
  be such that
  \begin{equation}
    I_{n-1}^{j} = \{m \in \supp(n-1) \;:\; \epsilon^m_t = + \;\;\forall j \le t \le i\},
  \end{equation}
  and $a^j = \sum_{r = j}^i n_r p^{(r)}-1$ so that the set $I_n^j$ is symmetric about $a^j$.

\vspace{1em}
\begin{center}
    \begin{tikzpicture}[scale=1, transform shape]
    \draw (-.5,0) -- (12.3,0);

    \fill (0,0) circle (0.05);
    \fill (.3,0) circle (0.05);

    \fill (.9,0) circle (0.05);
    \fill (1.2,0) circle (0.05);

    \fill (2.7,0) circle (0.05);
    \fill (3,0) circle (0.05);

    \fill (3.6,0) circle (0.05);
    \fill (3.9,0) circle (0.05);

    \fill (8.1,0) circle (0.05);
    \fill (8.4,0) circle (0.05);

    \fill (9,0) circle (0.05);
    \fill (9.3,0) circle (0.05);

    \fill (10.8,0) circle (0.05);
    \fill (11.1,0) circle (0.05);

    \fill (11.7,0) circle (0.05);
    \fill (12,0) circle (0.05);

    \draw [decorate,decoration={brace,amplitude=1}] (12.1,-0.2) -- (11.9,-0.2)
node [midway,yshift=-9]
{\small$I_{n}^0$};
    \draw [decorate,decoration={brace,amplitude=2}] (12.1,-0.8) -- (11.6,-0.8)
node [midway,yshift=-9]
{\small$I_{n}^1$};
    \draw [decorate,decoration={brace,amplitude=3}] (12.1,-1.4) -- (10.7,-1.4)
node [midway,yshift=-12]
{\small$I_{n}^2$};
    \draw [decorate,decoration={brace,amplitude=4}] (12.1,-2.0) -- (8.0,-2.0)
node [midway,yshift=-12]
{\small$I_{n}^3$};
    \draw [decorate,decoration={brace,amplitude=5}] (12.1,-2.6) -- (-.1,-2.6)
node [midway,yshift=-12]
{\small$I_{n}^4$};

    \draw (6,-0.2) -- (6,0.2);
    \node at (6, 0.3) {\tiny$a^4$};

    \draw (10.05,-0.2) -- (10.05,0.2);
    \node at (10.05, 0.3) {\tiny$a^3$};

    \draw (11.4,-0.2) -- (11.4,0.2);
    \node at (11.4, 0.3) {\tiny$a^2$};

    \draw (11.85,-0.2) -- (11.85,0.2);
    \node at (11.85, 0.3) {\tiny$a^1$};
  \end{tikzpicture}
\end{center}
\vspace{1em}

  We know that $U = \{m-1, m+1\} \cap \supp(n-1)$ is nonempty.
  Suppose that $\{m-1, m+1\} \cap I_n^j$ is either empty or $U$ for each $j$ and let $k$ be the maximum such that the intersection with $I_n^k$ is nonempty.
  Let $m'$ be the reflection of $m$ about $a^k$.
  Then for every $n > n' \ge m'$, we have $m' \in \supp(n') \Rightarrow m \in \supp(n')$.
  As a result, each composition factor, $D(n, n')$ for $m' \le n' < n$ of $S(n,m')$ occurs as a composition factor of $S(n, m)$.

  Further, $|\{m-1, m+1\} \cap \supp(n-1)| = |\{m'-1, m'+1\} \cap \supp(n-1)|$ by the symmetry of $I^k_{n-1}$ and hence every simple composition factor of $S(n,m')$ which upon restriction to $\TL_{n-1}$ has a trivial composition factor contributes also to the restriction of $S(n,m)$.
  We now only have two cases.
  If the trivial module is not a composition factor of $S(n,m')$, then we have accounted for all trivial factors of $S(n,m)\restr$ and so the trivial module is not a factor of $S(n,m')$.
  If the trivial module is a composition factor of $S(n,m')$ so that $m' \in \supp(n)$, then certainly $m \in \supp(n)$ and the multiplicity of the trivial module is bounded above by one.
  This ``missing multiplicity'' may be taken up by one of the factors of $S(n,m)$ of the form $D(n,j)$ for $m\le j < m'$, but by the argument in \cref{thm:lower}, we see that indeed it is a copy of the trivial.

  Finally, we must account for if $\{m-1, m+1\} \cap I_n^k \not \in \{U, \emptyset\}$ for some $k$.
  If this is the case, clearly $\{m-1, m+1\} \subseteq \supp(n)$ and $\{m-1, m+1\} \cap I_n^k = \{m+1\}$ and $m = a^{k+1}$.
  But then \cref{cor:when_simple} applies and $S(n,m)$ has no trivial composition factors as it is simple.
\end{proof}

We have thus proven our main theorem
\begin{theorem}\label{thm:main}
  The module $D(n,r)$ is a composition factor of $S(n,m)$ iff $m\in \supp(r)$, in which case it appears with multiplicity exactly one.
\end{theorem}
\begin{proof}
  Simply combine \cref{thm:upper} with \cref{cor:simple_mult_triv}.
\end{proof}

The majesty of \cref{thm:main} is that if $p = \infty$, so that we are in a characteristic zero field we see that $\supp(n) = \{n\}$ if $n\equiv_{\ell} -1$ and $\{n,n'\}$ where $n'$ and $n$ are reflections about the highest $\ell$-wall less than $n$ otherwise.
By comparison with the results of~\cite{ridout_saint_aubin_2014}, we see that \cref{thm:main} still holds.
If further, $\ell = \infty$ so the parameter $q$ is not a root of unity (or that $\delta$ never satisfies a quantum number), $\supp(n) = {n}$ and we recover the semi-simple case.

\section{Dimensions of Simple Modules}\label{sec:dimensions}
Recall \cref{prop:incremental_result_1}.
We are now able to fill in the remaining case of $m \equiv_\ell -1$.
\begin{lemma}\label{lem:dim_mod_m}
  If $m \equiv_\ell -1$ then
  \begin{equation}
    \dim D(n,m) = \dim S(n,m) - \sum_{\substack{m' > m \\ m \in \supp(m')}} \dim D(n,m')
  \end{equation}
\end{lemma}

This gives a recursive algorithm for the dimensions of all the simple modules of $\TL_n$ with any parameter over any field.
Compare \cref{fig:dimensions} which was computed with \cref{prop:incremental_result_1} and \cref{lem:dim_mod_m} with Table 7 in~\cite{andersen_2019}.

\begin{figure}[htpb]
  \centering\small%
  \begin{tabular}{cccccccccccc}
 $n\setminus m$          & 0   & 1   & 2    & 3    & 4    & 5    & 6    & 7   & 8    & 9   & 10 \\
    \hline
 0         & 1\\
 1         &     & 1\\
 2         & 1   &     & 1\\
 3         &     & 1   &      & 1\\
 4         & 1   &     & 3    &      & 1\\
 5         &     & 1   &      & 4    &      & 1\\
 6         & 1   &     & 9    &      & 4    &      & 1\\
 7         &     & 1   &      & 13   &      & 6    &      & 1\\
 8         & 1   &     & 27   &      & 13   &      & 7    &     & 1\\
 9         &     & 1   &      & 40   &      & 27   &      & 7   &      & 1\\
10         & 1   &     & 81   &      & 40   &      & 34   &     & 9    &     & 1\\
  \end{tabular}
  \caption{Dimensions of Simple Modules for $\ell = 3$ and $p=2$.}%
\label{fig:dimensions}%
\end{figure}

However, the formula above is recursive in nature.
Finding a closed formula is equivalent to inverting the decomposition matrix described in \cref{fig:in}:
\begin{equation}
  d_{n,m} = \begin{cases}
    1 & m \in \supp(n)\\
    0 & \text{else}
  \end{cases}
\end{equation}

Recall the definition of the relation $\triangleright$ from \cref{subsec:comb}.
We will say that $x \,\dot{\triangleright}\, y$ if $x \triangleright y$, $\nu_{(p)}(x) = \nu_{(p)}(y)$ and the $\nu_{(p)}(x)$-th digits of both $x$ and $y$ agree.  Here $\nu_{(p)}(x) = 0$ if $\ell \nmid x$ and $\nu_p(x/ \ell)+1$ otherwise.

Let
\begin{equation}\label{eq:enm}
  \tilde{e}_{n,m} =
  \begin{cases}
    0 & n \not \equiv_2 m \\
    -1 & \nu_{(p)}\left( \frac{n+m}{2} \right) > \nu_{(p)}(m) \text{ and } m \triangleleft \frac{n+m}{2} - 1\\
    1 & \nu_{(p)}\left( \frac{n+m}{2} \right) = \nu_{(p)}(m) \text{ and } m \,\dot{\triangleleft}\, \frac{n+m}{2}\\
    0 & \text{else}
  \end{cases}
\end{equation}

\begin{theorem}
  The matrix $(\tilde{e}_{n,m})$ is inverse to $(\tilde{d}_{n,m})$ where $\tilde{d}_{n,m}$ is 1 if $m \in \supp(n-1)$ and 0 else.
\end{theorem}

\begin{proof}
  It will suffice to show that $\sum_{m \in \supp(n-1)} e_{m,\tilde{n}} = \delta_{\tilde{n},n}$.
  We can thus leverage the structure of $\supp(n-1)$ to show the result.

  We first make the claim that if $n = \sum_{i =0}^I n_{i} p^{(i)}$ and $n_I \neq 0$, then for any $m < n_I p^{(I)}$, $e_{n,m} = - e_{n',m}$ where $n' = n_I p^{(I)} - \sum_{i=0}^{I-1} n_i p^{(i)}$.
  If so, note that for $n_{I}p^{(I)} \le m \le n$, we have that $e_{n,m} = e_{n-n_{I}p^{(I)},m-n_{I}p^{(I)}}$.
  Indeed, $m \triangleleft \frac{n+m}{2}-1 \Rightarrow m- n_{I}p^{(I)} \triangleleft \frac{n+m}{2}-1 - n_{I}p^{(I)}$ as we are simply removing the $I$-th digit on each side, and similarly for $m \,\dot{\triangleleft}\, \frac{n+m}{2}$.
  Since $\nu_{(p)}(m) \le I$, subtracting $n_{I}p^{(I)}$ maintains the valuation on both sides, and so we see the conditions of \cref{eq:enm} are maintained.
  The result then follows by a simple induction on $I$.

  It thus suffices to prove the claim.
  Suppose $e_{n,m} = 1$ and let $j = \nu_{(p)}(m) = \nu_{(p)}( \frac{n+m}{2})$.
  Then $\frac{n+m}{2} = x p^{(j)} + \sum_{i = j+1}^{I} x_i p^{(i)}$ and $m = x p^{(j)} + \sum_{i = j+1}^{I} y_i p^{(i)}$ for some digits $x_i \le y_i$.
  In this case, $\frac{n-m}{2} = \sum_{i=j+1}^{I}(x_i - y_i) p^{(i)}$ is a $(p,\ell)$-digit expansion and $n = xp^{(j)} + \sum_{i=j+1}^I (2x_i - y_i) p^{(i)}$.
  If so, $\frac{n-n'}{2}$ has valuation $j$ and least non-zero significant digit $x$ so $\frac{n'+m}{2} = \frac{n+m}{2}- \frac{n-n'}{2}$ has valuation at least $j$, but zero $j+1$-th digit so the valuation is strictly greater than $j$.
  Further, $\frac{n'+m}{2} = \sum_{i = j+1}^I(y_i-x_i) p^{(i)}$ and since $x_i \le p-1$ for each $i$ (except for $i=0$ where it is bounded by $\ell-1$), when writing out $\frac{n'+m}{2}-1$, we see that the $i$-th digit is larger (or equal to if $x_i = p-1$ and there is a carry) than $y_i$ so the dominance condition holds and $e_{n',m} = -1$.

  If on the other hand $e_{n',m} = -1$, let $\nu_{(p)}\left( \frac{n+m}{2}\right) = j$ and $\nu_{(p)}\left( m\right) = k < j$.
  Then write  $\frac{n+m}{2} = \sum_{i = j}^I x_i p^{(i)}$ and $m = \sum_{i=k}^I y_i p^{(i)}$.
  In this case $\frac{n-m}{2} = \sum_{i=k}^{j-1} (-y_i) p^{(i)} + \sum_{i=j}^I (x_i - y_i) p^{(i)}$ and so
  $n = \sum_{i=k}^{j-1}(-y_i)p^{(i)} + \sum_{i=j}^I(2x_i - y_i)p^{(i)}$.
  Thus we may write
  $\frac{n-n'}{2} = \sum_{i=k}^{j-1} (-y_i) p^{(i)} + \sum_{i=j}^I (2x_i - y_i) p^{(i)} - n_I p^{(I)}$ and so
  $\frac{n'-m}{2} = \sum_{i=k}^{j-1} y_i p^{(i)} + \sum_{i=j}^I (y_i - x_i) p^{(i)} + n_I p^{(I)}$.  It is clear that
$\nu_{(p)}(m) = k = \nu_{(p)}\left( \frac{n+m}{2} \right)$ and that the $k$-th through to $(j-1)$-st digits agree.
  It is hence sufficient to show now that $\sum_{i=j}^I y_i p^{(i)} \,\triangleleft\, \sum_{i=j}^I (y_i - x_i)p^{(i)} + n_I p^{(I)}$.
  Recall $m \,\triangleleft\, \frac{n+m}{2}-1$ so $\sum_{i=j}^I y_i p^{(i)} \,\triangleleft\, \sum_{i=j}^I x_i p^{(i)}$ since $j<k$.
  Hence $x_i \ge y_i$ and each digit of $\sum_{i=j}^I (y_i - x_i)p^{(i)} + n_I p^{(I)}$ is at least the corresponding digit of $m$ as desired.
\end{proof}

This directly leads to a closed form of the dimensions of the simple modules for all Temperley--Lieb algebras defined over fields.
\begin{corollary}\label{cor:dim_simple}
  The dimensions of the simple modules for $\TL_n^R(\delta)$ are given by
  \begin{equation}
  \dim D(n,m) = \sum_{r=0}^{\frac{n-m}{2}} \tilde{e}_{n-2r+1,m+1}\left({\binom{n}{r}} - \binom{n}{r-1}\right)
  \end{equation}
\end{corollary}
Note that this formula applies even in the semisimple case when $\delta$ is not a quantum root.
Here $\ell = \infty$ so that $(\tilde{e})$ is the identity matrix.

Further, this allows us to bound the dimensions from below in certain cases.
Naturally, if $\ell = \infty$ then we are in a semi-simple case and the simple module dimensions are exactly given by the cell module dimensions.
Thus, we may assume $\ell < \infty$ for the remainder of this section.

It will be useful to let $m_n$ be $\sqrt{n+2}-2$.
This value is chosen so that $\dim S(n,m) \ge \dim S(n,m-2)$ when $m \ge m_n$ and vice versa.

If $p = \infty$ then we know the structure of the cell modules exactly and can use this to give a lower bound on their dimension.
\begin{proposition}\label{prop:dim_bound_1}
If $p = \infty$, $2\ell-1 \le m$ and $r = (n-m)/2$,

\begin{equation*}
    \dim D(n,m) \ge \binom{n}{r}\frac{1}{(n-r+1)(n-r+2)}.
\end{equation*}
\end{proposition}
\begin{proof}
If $\ell \mid m + 1$ or $n - m \le \ell$ then $S(n,m)$ is irreducible and so 
\begin{align*}
  \dim D(n,m) = \dim S(n,m) &= \binom{n}{r} - \binom{n}{r-1} \\&= \binom{n}{r}\frac{n-2r+1}{n-r+1}\\&\ge\binom{n}{r}\frac{1}{(n-r+1)(n-r+2)}.
\end{align*}
Otherwise, it has two factors, $D(n,m)$ and $D(n,m')$ where if $m + 1 = [m_1, m_0]_{\ell, \infty}$ then $m' +1 = [m_1+1, \ell-m_0]_{\ell, \infty}$.  Note $m' \ge m +2$.

If, further, $m > m_n$,
\begin{align*}
  \dim D(n,m) &=\dim S(n,m) - \dim D(n,m')\\&\ge \dim S(n,m) - \dim S(n,m') \\&\ge \dim S(n,m) - \dim S(n,m+2)\\& = \frac{n!}{r!(n - r + 2)!}\big(n - (m+2)^2 - 2\big) \\&\ge \binom{n}{r}\frac{1}{(n-r+1)(n-r+2)}
\end{align*}

If, on the other hand, $2\ell \le m < m_n+4$ then let $m' + 1 = [m_1 - 1, \ell - m_0]_{\ell, \infty}$ and $m'' = m - 2\ell$.
Then by the structure of $S(n,m_1)$, we have the bound
\begin{align*}
    \dim S(n,m') &= \dim D(n,m') + \dim D(n,m) \\
                  &\le \dim S(n,m'') + \dim D(n,m)
\end{align*}
and so,
\begin{align*}
    \dim D(n,m) &\ge \dim S(n,m') - \dim S(n,m'') \\
                &\ge \dim S(n,m') - \dim S(n,m'-2) \\
                &= \frac{n!}{(r'+1)!(n - (r'+1) + 2)!}\big((m'-2)^2 - n + 2\big)
\end{align*}
where $m' = n-2r'$.
Notice that the bounds on $m$ ensure that $m'-2 < m_n$ and so the term $((m'-2)^2 - n + 2)$ is at least 1.  Hence, since $\lfloor n/2 \rfloor \ge r' \ge r + 1$
\begin{align*}
    \dim D(n,m) &\ge \binom{n}{r'+1}\frac{1}{(n-r'+1)(n-r'+2)}\\
     &\ge \binom{n}{r}\frac{1}{(n-r+1)(n-r+2)}.
\end{align*}
\end{proof}

We note that the condition $2\ell - 2 \le m$ is necessary. Indeed, if $\ell = 3$ so that $\delta = 1$ and $m = 1$ or $0$ depending on the parity of $n$, then $D(n,m)$ is given by the other one dimensional simple representation arising from the Temperley-Lieb monoid.
This is the representation where all diagrams act as unity.
It is clear that this dimension is constant over all $n$ and so does not obey an exponential bound.

In the modular case we have to use \Cref{cor:dim_simple}. We can rephrase it as
\begin{equation}\label{eq:dimension}
    \dim D(n,m) = \sum_{r = 0}^{\frac{n-m}{2}} e_{m + 2r+1, m+1} \dim S(n, m+2r)
\end{equation}
where
\begin{equation}
    e_{m+2r,m} =
  \begin{cases}
    -1 & \nu_{(p)}\left( m+r \right) > \nu_{(p)}(m) \text{ and } m \triangleleft m+r-1\\
    1 & \nu_{(p)}\left( m+r \right) = \nu_{(p)}(m) \text{ and } m \,\dot{\triangleleft}\, m+r\\
    0 & \text{else}
  \end{cases}
\end{equation}

This allows us to state a bound on the modular dimensions too.
\begin{proposition}\label{prop:dim_bound_2}
If $m \ge m_n$ and $r = (n-m)/2$, then
\begin{equation*}
    \dim D(n,m) \ge \binom{n}{r}\frac{1}{(n-r+1)(n-r+2)}.
\end{equation*}
\end{proposition}
\begin{proof}
Suppose that $e_{m+2r} = 1$ and write
\begin{align*}
    m &= [ m_k, \ldots, m_{s+1}, m_s, 0,\ldots, 0]_{\ell,p}\\
    m + r &= [ a_k, \ldots, a_{s+1}, m_s,0,\ldots, 0]_{\ell,p}
\end{align*}
where each $a_i \ge m_i$.  Then $r = [a_k-m_k, \ldots, a_{s+1}-m_{s+1},0,\ldots, 0]$ and 
\begin{equation*}
    m+2r = [2a_k - m_k, \ldots, 2a_{s+1}-m_{s+1}, m_s, 0, \ldots, 0]
\end{equation*}
so in particular, the reflection of $m+2r$ about the wall above is $m + 2r + 2t$ for $t = (p-m_s)p^{(s)}$.  If we write this as $m + 2r'$ then 
\begin{align*}
m + r' &= m + r + t\\ &= [a_k, \ldots, a_{s+1}, p-m_s+m_s,0,\ldots, 0]_{\ell, p}\\ &= [a_k, \ldots, a_{s+1}+1, 0,\ldots, 0]_{\ell, p}
\end{align*}
and
\begin{align*}
m + r'-1 = [a_k, \ldots, a_{s+1}, p-1,\ldots,p-1, \ell-1]_{\ell, p}
\end{align*}
so $e_{m+2r'} = -1$.

On the other hand, suppose that $e_{m+2r} = -1$ and write
\begin{align*}
    m &= [ m_k, \ldots, m_{s+1}, m_s, 0,\ldots, 0]_{\ell,p}\\
    m + r &= [ a_k, \ldots, a_{s+1}, m_{s'},0,\ldots, 0]_{\ell,p}
\end{align*}
for some $s' > s$.
Then, since $m \triangleleft m + r - 1$, we must have that $m_i \le a_i$ and $m_{s'} \le a_{s'}-1$.  We note that
\begin{equation*}
    r = [a_k - m_k, \ldots, a_{s'} - m_{s'} - 1, p-m_{s'-1}, \ldots, p-m_s, 0,\ldots, 0]
\end{equation*}
so that
\begin{equation*}
    m + 2r = [2a_k - m_k, \cdots\cdots, p-m_s, 0,\ldots, 0].
\end{equation*}
Here we have ommitted the majority of the coefficient as we only require the least significant digit to calculate the reflection over the wall below.  Here, we see that for $t = (p-m_s)p^{(s)}$, we seek $m + 2r' = m+2r-2t$ so that
\begin{align*}
m + r' = [a_k, \ldots, a_{s'}-1, p-1, \ldots, p-1, m_s, 0,\ldots, 0]
\end{align*}
and so $e_{m+2r', m} = 1$.

Thus, for a given $m$, we have paired off all the nonzero coefficients appearing in \cref{eq:dimension}.
If $m > m_n$, we thus have that $\dim S(n,m+2r) > \dim S(n,m+2r')$ and so the sum of each of these pairings contributes a non-negative term to the sum.  In particular, if $m'$ is the reflection of $m$ over the wall above $m$, we have shown
\begin{align*}
    \dim D(n,m) &\ge \dim S(n,m) - \dim S(n,m')\\
    &\ge \dim S(n,m) - \dim S(n,m+2)\\
    &=\frac{n!}{r!(n - r + 2)!}\big(n - (m+2)^2 - 2\big)\\
    &\ge \binom{n}{r}\frac{1}{(n-r+1)(n-r+2)}.
\end{align*}
\end{proof}
Note that we have the same bound in both \Cref{prop:dim_bound_1} and \Cref{prop:dim_bound_2}.
\begin{figure}[htpb]
\begin{center}
  \includegraphics[width=0.8\textwidth]{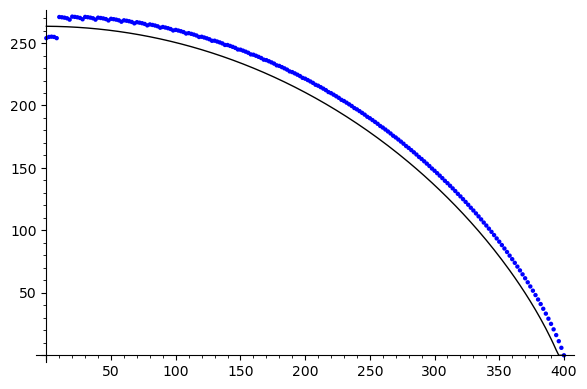}
\end{center}
\caption{The dimensions of the simple modules of $\TL_{400}$ under characteristic zero with $\ell = 10$ plotted on a logarithmic scale. The dimensions are plotted in blue and the bound given by the propositions is in black. Note that the bound fails for sufficiently small $m$.}%
\label{fig:sierpinski}
\end{figure}

\section{Applications and Prospects}\label{sec:applications}
We highlight a few applications of the above analysis of the Temperley-Lieb algebra over arbitrary fields.

\subsection{$p$-Kazhdan-Lusztig polynomials}
In~\cite{jensen_williamson_2015}, Jensen and Williamson define the $p$-canonical basis of a Hecke algebra of a crystallographic Coxeter system.
Elias shows that the category of Soergel Bimodules for a dihedral Coxeter group is equivalent to that of the Temperley--Lieb algebra.
Moreover, the rank of the ``local intersection form'' is exactly that of the Gram matrix mentioned in \cref{sec:cell-modules}.
Hence \Cref{prop:incremental_result_1} and \Cref{lem:dim_mod_m} compute the $p$-canonical basis of the Hecke algebra of the non-crystallographic dihedral groups for any $p$ and over any (not neccessarily integral) realisation.

\subsection{Steinberg Maps and $\ell$-Dilation}
If $n \equiv_\ell -1$ then
the cell modules indexed by $m \equiv_\ell -1$ obey a recursion relation.
Let $d_\ell(t) = (t+1)/\ell-1$.
\begin{theorem}
  The simple module $D(n, m')$ is a composition factor of $S(n,m)$ iff $D(d_\ell(n), d_\ell(m'))$ is a composition factor of $S(d_\ell(n), d_\ell(m))$.
\end{theorem}
\begin{proof}
  This is a simple application of \cref{thm:upper}.
\end{proof}

Let $\TL_n'\subseteq \TL_n$ be the direct sum of all the projective covers of such $S(n,m)$ in $\TL_n$.
Note that this is possibly not a single block (if, for example, the trivial module is projective), but is the sum of some number of blocks of $\TL_n$.
As such it is an associative algebra (although not a subalgebra of $\TL_n$ as its unit differs).

\begin{conjecture}
  $\TL_n'$ is Morita equivalent to $\TL_{d_\ell(n)}$.
\end{conjecture}

In the representation theory of $U_q(\mathfrak{sl}_2)$, there is a Steinburg functor on modules.
It is currently unclear what the translation of this functor to the Temperley-Lieb theory is, however it is expected to relate the theory of $\TL_n$ to that of $\TL_{d_\ell(n)}$ when $\ell = p$.

\subsection{Jones-Wenzl idempotents}
The Jones-Wenzl idempotent is a celebrated element of $\TL^{\Q(\delta)}_n(\delta)$.
It is the unique idempotent $\JW_n$ such that $u_i\cdot \JW_n = 0$ for all $1\le i<n$.
Its existance is equivalent to the statement that the trivial $\TL^{\Q(\delta)}_n(\delta)$ module is projective.

As can be seen from \cref{fig:in}, this is no longer the case in characteristic zero, positive characteristic or mixed characteristic.
In fact, it only occurs if $\supp(n) = \{n\}$ --- that is if $n <\ell$ or $n = ap^{(i)}-1$ for $1\le a < p$.
If so, we can use idempotent lifting techniques to show that the idempotent defining the projective cover of the trivial is in fact the reduction of $\JW_n$.

This answers the question ``when is the Jones-Wenzl idempotent defined'' over mixed characteristic to be that where $n <\ell$ or $n = ap^{(i)}-1$ for $1\le a < p$.
Equivalently, by comparing with \cref{cor:nonzero_binoms} we see that this is equivalent to all the quantum binomials $\gaussianquant{n}{r}$ being nonzero for $0\le r\le n$.

This fact has been shown by Webster in the appendix of~\cite{elias_libedinsky_2017}, however his construction employs the use of linear endomorphisms of $U_q(\mathfrak{sl}_2)$ modules which do not exist in the Temperley-Lieb algebra.
This is the first proof of which the author is aware that is purely diagrammatic.

\section*{Acknowledgements}
\noindent
The author would like to thank Dr Stuart Martin for his support and Alexis Langlois--R\'emillard for comments on an earlier draft.
The author is supported by a DTP studentship from the Department of Pure Mathematics and Mathematical Sciences of the University of Cambridge funded by the Engineering and Physical Sciences Research Council.

%
%

\bibliographystyle{alpha}
\bibliography{all_cite}

\end{document}